\documentclass{article}
\usepackage{amsmath}
\usepackage{amssymb}
\usepackage{amsfonts}

\newtheorem{theorem}{Theorem}

\newtheorem{corollary}[theorem]{Corollary}

\newtheorem{definition}[theorem]{Definition}
\newtheorem{example}[theorem]{Example}

\newtheorem{lemma}[theorem]{Lemma}
\newtheorem{notation}[theorem]{Notation}

\newtheorem{proposition}[theorem]{Proposition}
\newtheorem{remark}[theorem]{Remark}

\newenvironment{proof}[1][Proof]{\noindent\textbf{#1.} }{\ \rule{0.5em}{0.5em}}
\input{tcilatex}
\begin{document}

\title{Every continuum has a compact universal cover}
\author{Conrad Plaut \\
Department of Mathematics\\
The University of Tennessee\\
Knoxville TN 37996\\
cplaut@utk.edu}
\maketitle

\begin{abstract}
We define the compact universal cover of a compact, metrizable connected
space (i.e. a continuum) $X$ to be the inverse limit of all continua that
regularly cover $X$. We show that such covers do indeed form an inverse
system with bonding maps that are regular covering maps, and the projection $%
\widehat{\phi }:\widehat{X}\rightarrow X=\widehat{X}/\pi _{P}(X)$ of the
inverse limit is a generalized regular covering map, where $\widehat{X}$ is
a continuum that is \textquotedblleft compactly simply
connected\textquotedblright\ in the sense that $\pi _{P}(\widehat{X})=1$. We
call $\pi _{P}(X)$ the \textquotedblleft profinite fundamental
group\textquotedblright\ of $X$. We prove a Galois Correspondence for closed
normal subgroups of $\pi _{P}(X)$, uniqueness, universal and lifting
properties. As an application we prove that every non-compact manifold that
regularly covers a compact manifold has a unique \textquotedblleft profinite
compactification\textquotedblright , i.e. an imbedding as a dense subset in
a compactly simply connected continuum. As part of the proof of
metrizability of $\widehat{X}$ we show that every continuum has at most $n!$
non-equivalent $n$-fold covers by continua.
\end{abstract}

\section{Introduction}

Efforts to generalize the powerful tools of covering space theory beyond
Poincar\'{e} spaces (connected, locally path connected and semilocally
simply connected spaces) go back at least to the work of Kawada in 1950 (%
\cite{Kaw}). Kawada defined a notion of generalized covering map for
topological groups to be a homomorphism having conditions, including totally
disconnected kernel, that permit covering maps that are not locally 1-1.
Note that his definition was too weak because he did not require that the
kernel be a complete topological group, and Berestovskii-Plaut showed in 
\cite{BPTG} that Kawada's uniqueness statement for generalized universal
covers is incorrect. The situation is remedied in \cite{BPTG} by in effect
defining generalized covering homomorphims to be inverse limits of regular
covering homomorphisms in the traditional sense, so that the kernel of the
resulting map is prodiscrete and hence complete. This led to the theorem
that every compact, connected, metrizable group has a unique generalized
universal cover by a compact group (\cite{BPCG}), which is a product of
connected, simply connected, compact simple Lie groups and universal
solenoids. This \textit{compact universal covering group} is in fact the
inverse limit of all traditional covering homomorphisms of the group by
finite groups, and therefore \cite{BPCG} is a kind of precursor to the
present paper. 

Efforts that do not involve the extra structure of a topological group go
back at least to Fox in the early 1970's (\cite{F}), in which he
strengthened the notion of covering map to what he called an \textit{overlay}%
. Yet finding a \textit{universal} cover in greater generality certainly
requires \textit{weakening }the notion of covering map in natural ways that
retain the most significant properties of covering maps: lifting properties;
a universal property in a categorical sense (and hence uniqueness); the fact
that the universal covering map map is the quotient via a group action by a
(topological) group that generalizes the topologically discrete traditional
fundamental group; and a Galois correspondence between (topologically
closed) normal subgroups of that group and generalized regular covering maps
of the original space.

Two independent efforts proving the existence of generalized universal
covering maps were published in the same year by Fischer-Zastrow (\cite{FZ})
and Berestovskii-Plaut (\cite{BPUU}). The results of \cite{FZ} require that
the space in question be path and locally path connected, whereas the
results of \cite{BPUU} are applicable to many spaces without these
properties. Those spaces include the Warsaw Circle and even some totally
disconnected spaces, but not, for example, solenoids. 

The first (dyadic) solenoid was introduced by Vietoris in 1927 (\cite{V})
and these objects were extensively generalized by McCord in 1965 (\cite{Mc}%
). McCord introduced \textit{solenoidal spaces }as inverse limits of
traditional regular coverings of a Poincar\'{e} space parameterized by the
natural numbers $\mathbb{N}$, and showed that the resulting projection map
is a principle fiber bundle. Solenoidal spaces gained new prominence as a
result of the work of Smale, who showed that solenoids appear as expanding
attractors in dynamical systems in 1967 (\cite{Sm}). Solenoidal spaces and
their generalizations (e.g. matchbox manifolds) continue to play a
significant role in dynamical systems--see for example \cite{H1} and \cite%
{H2}. Given their importance, a generalized covering space theory that
includes solenoidal spaces is desirable. 

In this paper we define the \textit{compact universal cover} $\widehat{\phi }%
:\widehat{X}\rightarrow X$ of \textit{any} continuum (compact, connected,
metrizable space) $X$ to be the projection from the inverse limit $\widehat{X%
}$ of all continua $Y$ such that there is a traditional regular covering map 
$f:Y\rightarrow X$. We call such an $f$ a \textit{cover by a continuum}.
This construction fits into a broader framework of generalized regular
covering maps between uniform spaces (\cite{PQ}), which we show in the
present paper is (unlike traditional regular covering maps between
topological spaces), a category (Theorem \ref{compo}). We will give more
precise details later, but a \textit{generalized regular covering map}
(Definition \ref{guc}) between uniform spaces is the quotient map of a 
\textit{prodiscrete and isomorphic} action of a group of uniform
homeomorphisms, and in particular it is generally not a local homeomorphism.
Equivalently, a generalized regular covering map resolves as an inverse
system of \textit{discrete covers}, which are the uniform analog of
traditional regular covering maps.

To simplify our notation we will denote regular covers of a continuum $X$ by
continua by $f_{i}:X_{i}\rightarrow X/G_{i}$. Of course it is possible (e.g.
with the circle) that for distinct $f_{i},f_{j}$, $X_{i}$ is homeomorphic to 
$X_{j}$, but by assumption if $i\neq j$ then $f_{i}$ and $f_{j}$ are not
equivalent.

Basepoints play an important role in many of our proofs, but in the end our
results are basepoint independent up to equivalence. Therefore we will avoid
the cumbersome \textquotedblleft pointed\textquotedblright\ notation. When
basepoints are involved we will always denote them by $\ast $, and unless
otherwise stated, all maps are assumed to be basepoint preserving. This
approach also naturally leads to stating uniqueness of maps in terms of
equivalence classes relative to an action rather than involving basepoints.
For example, suppose that we have quotients via actions of bijections $\pi
_{1}:X_{1}\rightarrow X=X_{1}/G_{1}$ and $\pi _{2}:X_{2}\rightarrow
X=X_{2}/G_{2}$ and a function $f:X_{1}\rightarrow X_{2}$ with some property
(e.g. uniformly continuous) such that $\pi _{2}\circ f=\pi _{1}$. When we
say $f$ is the unique (mod $G_{2}$) function with that property such that $%
\pi _{2}\circ f=\pi _{1}$, we mean that the set of functions with those
properties is precisely the set of functions of the form $g\circ f$, where $%
g\in G_{2}$.

\begin{definition}
Given two covers by continua $f_{i}:X_{i}\rightarrow X$ and $%
f_{j}:X_{j}\rightarrow X$, we say that $f_{i}\leq f_{j}$ if there is a
continuous onto map $f_{ij}:X_{j}\rightarrow X_{i}$ such that $%
f_{j}=f_{i}\circ f_{ij}$.
\end{definition}

\begin{theorem}[Existence]
\label{dir}For any cotinuum $X$,

\begin{enumerate}
\item The set of all covers $f_{i}:X_{i}\rightarrow X_{i}/G_{i}$ is a
directed set.

\item For all $i,j$ there are unique (mod $G_{i}$) regular covering maps $%
f_{ij}:X_{j}\rightarrow X_{i}$ that are the bonding maps for the inverse
system, and compatible homomorphisms $\theta _{ij}:G_{j}\rightarrow G_{i}$,
i.e. for all $i,j$ and $g\in G_{j}$, $\theta _{ij}(g)\circ
f_{ij}=f_{ij}\circ g$.

\item The projection $\widehat{\phi }:\widehat{X}=\underleftarrow{\lim }%
X_{i}\rightarrow X=\widehat{X}/\pi _{P}(X)$ is a generalized regular
covering map with profinite deck group $\pi _{P}(X)=\underleftarrow{\lim }%
G_{i}$, and $\widehat{X}$ is a continuum.
\end{enumerate}
\end{theorem}

We call $\pi _{P}(X)$ the \textit{profinite fundamental group of }$X$. When $%
\pi _{P}(X)=1$ we say that $X$ is \textit{compactly simply connected}. Note
that for a compact manifold $M$, $\widehat{M}$ is by definition a
\textquotedblleft weak solenoid\textquotedblright --see \cite{HL}. For
example, when $X$ is the circle, $\widehat{X}$ is by definition the
so-called universal solenoid.

\begin{theorem}[Galois Correspondence]
\label{corres}If $f:Y\rightarrow X=Y/G$ is a generalized regular covering
map by a continuum then $f$ is equivalent to the induced quotient $\pi :%
\widehat{X}/K\rightarrow \left( \widehat{X}/K\right) /\left( \pi
_{P}(X)/K\right) $ for some closed normal subgroup $K$ of $\pi _{P}(X)$.
Conversely any such induced quotient for a closed normal subgroup of $\pi
_{P}(X)$ is a generalized regular covering map of $X$ by a continuum.
\end{theorem}

\begin{theorem}[Universal Property]
\label{main}If $f:Y\rightarrow X=Y/G$ is a generalized regular covering map
by a continuum then there is a unique (mod $G$) generalized regular covering
map $f_{L}:\widehat{X}\rightarrow Y$ called the lift of $f$ such that $%
\widehat{\phi }=f\circ f_{L}$.
\end{theorem}

\begin{theorem}[Uniqueness]
\label{unique}$\widehat{X}$ is compactly simply connected, and if $%
f:Y\rightarrow X=Y/G$ is a generalized universal cover by a compactly simply
connected continuum $Y$, then $f$ is equivalent to $\widehat{\phi }$.
\end{theorem}

\begin{corollary}
\label{ecor}If $f:Y\rightarrow X=Y/G$ is a generalized regular covering map
by a continuum then there is a unique (mod $\pi _{P}(X)$) homeomorphism $h:%
\widehat{X}\rightarrow \widehat{Y}$ such that $\widehat{\phi }_{X}=f\circ 
\widehat{\phi }_{Y}\circ h$, where $\widehat{\phi }_{X}$ and $\widehat{\phi }%
_{Y}$ are the compact universal covering maps of $X$ and $Y$, respectively.
\end{corollary}

When $X$ is a compact metrizable Poincar\'{e} space, following (\cite{Mc}),
the universal cover $\widetilde{X}$ embeds as a dense subgroup of what we
are now calling $\widehat{X}$ (see also Theorem \ref{snark}). More
generally, given any injective homomorphism $\eta $ of $\pi _{1}(X)$ into a
profinite group $P=\underleftarrow{P_{i}}$, we obtain a generalized regular
covering map $f:Y\rightarrow X=Y/P$ by taking the compact covers
corresponding to the kernels of $\eta _{i}:=\pi ^{i}\circ \eta $, where $\pi
^{i}$ is the projection onto the finite group $P_{i}$. However, Theorem \ref%
{unique} gives us something new: this embedding of $\widetilde{X}$ into a
compactly simply connected continuum is independent of the particular space $%
X$ that $\widetilde{X}$ covers.

\begin{corollary}[Profinite Compactification]
\label{PC}Suppose that $Y$ is a metrizable simply connected Poincar\'{e}
space such that there is a regular covering map $f:Y\rightarrow X$ for some
compact space $X$. Then $Y$ embeds as a dense subgroup in a unique compactly
simply connected continuum $\overline{Y}$.
\end{corollary}

\begin{theorem}
\label{lifting}Let $f:Z\rightarrow X$ be a continuous function such that $X$
is a continuum and $Z$ is a simply connected Poincar\'{e} space. Then there
is a unique (mod $\pi _{P}(X)$) uniformly continuous function $%
f_{L}:Z\rightarrow \widehat{X}$ called the lift of $f$ such that $f=\widehat{%
\phi }\circ f_{L}$.
\end{theorem}

The preceding theorem simply follows from the traditional lifting property
of regular covers and the universal property of the inverse limit. This
result leads to leads to a homomorphism from $\pi _{1}(X)$ into the
profinite group $\pi _{P}(X)$, which for Peano continuua has dense image.

We will now sketch out our proofs. The first problem is to show that covers
by continua form a directed set. This problem has a classical solution when $%
X$ is a Poincar\'{e} space, and hence has a traditional universal cover. In
this case one has the traditional Galois Correspondence between regular
covering maps of a Poincar\'{e} space and induced quotients via normal
subgroups of the fundamental group--which in the case of compact covers have
finite index. Therefore the partial order of regular covering maps
corresponds to reverse inclusion of normal subgroups of $\pi _{1}(X)$. Since
the intersection of any two normal subgroups of finite index is also a
normal subgroup of finite index, the corresponding covering spaces form a
directed set.  Note that the proof of the latter fact depends on the
existence of a traditional universal cover, but the resolution of
generalized regular covering maps as inverse limits of discrete covers shows
that they are always principle fibrations (but we have no present use for
this observation--being a generalized regular covering map is much stronger).

For a general continuum there is no suitable analog of the classical
universal cover; instead, we first prove (Theorem \ref{quoteq}) a version of
Galois Correspondence for the map $\phi _{E}:X_{E}\rightarrow X$ introduced
by Berestovskii-Plaut in \cite{BPUU}, where $X$ is a uniform space and $E$
is an entourage in $X$. We will recall some basics about uniform spaces
later, but roughly speaking an entourage is a symmetric neighborhood of the
diagonal in $X\times X$ that allows one to uniformly measure
\textquotedblleft closeness\textquotedblright\ in the space. When $X$ is
compact there is a unique uniform structure compatible with the topology and 
$E$ may be taken to be any symmetric neighborhood of the diagonal in $%
X\times X$. The map $\phi _{E}:X_{E}\rightarrow X$, defined in \cite{BPUU}
using discrete homotopy theory, is a kind of \textquotedblleft universal
covering map at the scale of $E$\textquotedblright . It is a quotient map
via a discrete group $\pi _{E}(X)$ that is a kind of \textquotedblleft
fundamental group at the scale of $E$\textquotedblright . The set of all $%
X_{E}$ with natural bonding maps $\phi _{EF}:X_{F}\rightarrow X_{E}$ when
the entourage $F$ is contained in $E$, is called the \textit{fundamental
inverse system} of $X$. The projection $\phi :\widetilde{X}\rightarrow X$ of
the inverse limit of this system is, for suitably nice spaces (including
Peano continua), a generalized regular covering map called the \textit{%
uniform universal cover} (\cite{BPUU}, \cite{PWC}). However, this
construction fails to produce a surjective projection, hence a generalized
regular covering map, for more exotic continua such as solenoids.

In this paper we show that if $f:Y\rightarrow X$ is a cover by a continuum,
all sufficiently small entourages $E$ in $X$ are \textquotedblleft evenly
covered\textquotedblright\ in a natural sense, and for such $E$, $f$ is
equivalent to the induced regular covering map $\pi :X_{E}/K\rightarrow
\left( X_{E}/K\right) /(\pi _{E}(X)/K)$ for some uniquely determined normal
subgroup $K$ of finite index in $\pi _{E}(X)$ (Theorem \ref{quoteq}).
However, there are complications: $X_{E}$ may be neither compact nor
connected--in fact when $X$ is a solenoid, $X_{E}$ may have uncountably many
components. The fact that $X_{E}$ is not compact plunges us firmly into the
world of uniform spaces and requires development of a generalized covering
space theory in this broader context--which as mentioned above also has the
advantage of being a category. In order to tie these results to the
traditional covering space theory for metrizable spaces, we show that given
a traditional regular covering map between metrizable spaces, the spaces can
be \textquotedblleft uniformized\textquotedblright\ so that the covering
maps is a discrete cover of the resulting uniform spaces (Theorem \ref%
{uniformize}). In this way, the results of this paper are applicable more
generally to metrizable spaces. To see why this is not a contradiction to
the fact that the composition of traditional regular covering maps may not
be a regular covering map, see Example \ref{covex}. A key tool in this paper
is the General Chain Lifting Lemma (Lemma \ref{cl}), the discrete analog of
the path and homotopy lifting properties from the traditional theory.

As for lack of connectedness, the first problem is that the traditional
notion of equivalence of regular covering maps is too weak for covers by
spaces that may not be connected. Instead we use a very natural notion of
equivalence of uniform group actions (Definition \ref{equivdef}), which is
stronger in general but is equivalent to the traditional definition for
generalized regular covering maps between connected spaces, see Example \ref%
{strong} and Corollary \ref{eqeq}. The maps $\phi _{E}:$ $X_{E}\rightarrow X$
are also what we call \textquotedblleft proper discrete
covers\textquotedblright\ (Definition \ref{pdc}) which have a weaker
connectivity property that is sufficient to apply the General Chain Lifting
Lemma. Finally, we show that when $K$ is a normal subgroup of $\pi _{E}(X)$
of finite index, $X_{E}/K$ is compact with finitely many components, and one
may then restrict to any component to obtain a regular covering map by a
continuum (Proposition \ref{cpt}).

There remains one final complication, to find a countable cofinite inverse
system. Doing so allows us to both show that $\widehat{\phi }:\widehat{X}%
\rightarrow X$ is surjective, hence a generalized regular covering map, and
show that $\widehat{X}$ is metrizable (classical results on inverse systems
show that $\widehat{X}$ is compact and connected). The existence of this
system is a consequence of the theorem below, the proof of which requires
some delicate arguments involving discrete homotopies (see the proof of
Theorem \ref{corr}):

\begin{theorem}
\label{factor}If $X$ is a continuum then there are at most $n!$
(non-equivalent) regular $n$-fold covering maps of $X$ by continua.
\end{theorem}

In this context we note that Gumerov (\cite{G}) gave necessary and
sufficient conditions for existence of an $n$-fold covering map between
solenoids. 

\section{Uniform spaces and discrete homotopy theory}

We will use the following notation in this paper if $W\subset X\times X$ is
any symmetric subset containing the diagonal. The $W$-ball in $X$ at $x$ is $%
B(x,W)=\{y\in X:(x,y)\in W\}$. In a metric space we have the traditional
metric balls $B(x,\varepsilon )=B(x,E_{\varepsilon })$, where $%
E_{\varepsilon }:=\{(x,y):d(x,y)<\varepsilon \}$ is the \textquotedblleft
metric entourage\textquotedblright\ for $\varepsilon $. Recall that a
uniform space consists of a topological space $X$ together with a collection
of symmetric neighborhoods of the diagonal in $X\times X$ called \textit{%
entourages}. The set of entourages is called the \textit{uniform structure}
or \textit{uniformity} of $X$. The characterizing properties are the
\textquotedblleft triangle inequality\textquotedblright : for every
entourage $F$ there is an entourage $E$ such that $E^{2}\subset F$; along
with the fact that every symmetric set containing an entourage is an
entourage. The uniformity is compatible with the topology if whenever $U$ is
open in $X$ and $x\in U$ there is some entourage $E$ such that $%
B(x,E)\subset U$.

The set $E^{k}$ can be equivalently described using the notion of an $E$%
-chain, which is a finite sequence $\alpha =\{x_{0},...,x_{n}\}$ such that
for all $i$, $(x_{i},x_{i+1})\in E$. So $E^{k}$ is the set of all $(x,y)$
such that there is an $E$-chain $\{x=x_{0},...,x_{k}=y\}$. The canonical
examples of uniform spaces are metric spaces, with \textit{metric entourages}
$E_{\varepsilon }$ forming a basis for the uniformity; compact topological
spaces, which have a unique uniform structure compatible with the topology,
in which entourages are simply all symmetric neighborhoods of the diagonal;
and topological groups. Given a topological group and a basis for the
topology at the identity $1$, there are uniquely determined
\textquotedblleft left\textquotedblright\ and \textquotedblleft
right\textquotedblright\ uniformities that are invariant with respect to
left or right translation. That is, for any $U$ in the basis at $1$, the
corresponding entourage $E_{U}$ in the left uniformity consists of all $%
(g,h) $ such that $g^{-1}h\in U$. So if $k\in G$ then since $\left(
kg\right) ^{-1}kh=g^{-1}h$, $E_{U}$ is invariant with respect to left
multiplication.

We will use simplified notation for Cartesian products of functions, such as
using $g(W)$ to denote $\left( g\times g\right) (W)$. For example, $W$ is
invariant with respect to a bijection $g:X\rightarrow X$ by definition if $%
g(W)=W$. With this notation, $f:X\rightarrow Y$ is \textit{uniformly
continuous} if and only if for every entourage $E$ in $Y$, $f^{-1}(E)$ is an
entourage in $X$. We say that $f$ is \textit{bi-uniformly continuous} if in
addition $f$ is surjective and $f(E)$ is an entourage in $Y$ for every
entourage $E$ in $X$. A 1-1 bi-uniformly continuous function is called a 
\textit{uniform homeomorphism}. In the compact case, with the unique
compatible uniformity, continuity and uniform continuity are equivalent.

Uniform spaces are the appropriate setting for discrete homotopy theory as
developed in \cite{BPTG}, \cite{PQ}, \cite{BPUU}, and \cite{PWC}. We recall
some basics now. Given an $E$-chain $\alpha =\{x_{0},...,x_{n}\}$ in a
uniform space, a \textit{basic move} consists of adding or removing a single
point, except either endpoint, so that the resulting chain is still an $E$%
-chain. An $E$-homotopy of an $E$-chain is a finite sequence of basic moves.
The $E$-homotopy equivalence class of $\alpha $ is denoted by $[\alpha ]_{E}$%
. We will frequently use without reference the fact that if $f:X\rightarrow
Y $ is a (possibly not continuous!) function between uniform spaces and $E,F$
are entourages in $X,Y$, respectively, such that $f(E)\subset F$ then given
any $E$-homotopic $E$-chains $\alpha $ and $\beta $ in $X$, $f(\alpha )$ and 
$f(\beta )$ are $F$-homotopic $F$-chains.

If $\alpha $ is an $E$-loop that is $E$-homotopic to the trivial chain
(consisting of its start/end point) then $\alpha $ is called $E$-null.
Fixing a basepoint $\ast $, the set of all $[\alpha ]_{E}$ such that $\alpha 
$ starts at $\ast $ is denoted by $X_{E}$, and $\phi _{E}:X_{E}\rightarrow X$
is the endpoint map. For any entourage $F\subset E$, in \cite{BPUU} we
define $F^{\ast }$ to be the set of all ordered pairs $([\alpha ]_{E},[\beta
]_{E})$ such that $[\overline{\alpha }\ast \beta ]_{E}=[\{a,b\}]_{E}$ and $%
(a,b)\in F$, where $a$ and $b$ are the endpoints of $\alpha $ and $\beta $,
respectively. Here $\ast $ denotes concatenation of chains when the endpoint
of the first chain is the first point of the second chain, and $\overline{%
\alpha }$ is the reversal of the chain $\alpha $. Note that $\{a,b\}$ must
be an $E$-chain since $F\subset E$. There is a useful equivalent definition
of $F^{\ast }$ (which was the original definition in \cite{BPUU}), namely $%
\left( [\alpha ]_{E},[\beta ]_{E}\right) \in F^{\ast }$ if and only if we
may write (up to $E$-homotopy) $\alpha $ as an $E$-chain $\{\ast
=x_{0},...,x_{n},a\}$ and $\beta $ as an $E$-chain $\{\ast
=x_{0},...,x_{n},b\}$ with $(a,b)\in F$. Equivalently, $(a,b)\in F$ and we
may write $\beta $ as $\alpha \ast \{a,b\}=\{\ast =x_{0},...,x_{n},a,b\}$.

In what follows we will not use the extra brackets $\{\}$ in our notation,
e.g. writing $[a,b]_{E}$ rather than $[\{a,b\}]_{E}$. The set of all $%
F^{\ast }$ is a basis for a uniformity on $X_{E}$. The set of all
equivalence classes $[\lambda ]_{E}$ of $E$-loops based at $\ast $ is
denoted by $\pi _{E}(X)$, which is a group with product induced by
concatenation of chains. The group $\pi _{E}(X)$ acts on $X_{E}$ by uniform
homeomorphisms via preconcatentation: $[\lambda ]_{E}([\alpha
]_{E})=[\lambda \ast \alpha ]_{E}$, and the entourages $F^{\ast }$ for $%
F\subset E$ are invariant with respect to this action. Then $\phi
_{E}:X_{E}\rightarrow X=X_{E}/\pi _{E}(X)$ is the quotient map. Moreover,
the restriction of $\phi _{E}$ to any $B([\alpha ]_{E},F^{\ast })$ is a
bijection onto $B(a,F)$, where $a$ is the endpoint of $\alpha $.

When $F\subset E$ is an entourage, since $F^{\ast }$ is an entourage in the
uniform space $X_{E}$, it is natural to consider the space $(X_{E})_{F^{\ast
}}$. Proposition 23 in \cite{BPUU} naturally identifies $X_{F}$ with $%
(X_{E})_{F^{\ast }}$ for any $F\subset E$ via the mapping 
\begin{equation}
\lbrack x_{0},...,x_{n}]_{F}\mapsto \lbrack \lbrack
x_{0}]_{E},[x_{0},x_{1}]_{E},...,[x_{0},...x_{n}]_{E}]_{F^{\ast }}\text{.}
\label{clf}
\end{equation}%
This formula is useful in various ways. For now, note that when $F=E$ this
implies that the map $\phi _{E^{\ast }}:(X_{E})_{E^{\ast }}\rightarrow X_{E}$
is a uniform homeomorphism, hence every $E^{\ast }$-loop in $X_{E}$ is $%
E^{\ast }$-null (this explicit statement has not been made previously, but
this argument was used in the proof of Theorem 77 in \cite{BPUU}).

Berestovskii-Plaut also showed that if $X$ is path connected and uniformly
semilocally simply connected (we will call such $X$ a \textquotedblleft
uniform Poincar\'{e} space\textquotedblright ) then for suffiently small
entourages $E$ having connected balls (and such always exist), $X_{E}$ is
the universal cover of $X$ and $\pi _{E}(X)=\pi _{1}(X)$. In general the
adjective \textquotedblleft uniformly\textquotedblright\ preceding a
traditional local topological property means that there are arbitrarily
small entourages $E$ such that the $E$-balls have the property. In compact
spaces, local properties are automatically uniformly local with respect to
the unique compatible uniform structure (see for example Proposition 68 in 
\cite{BPUU}).

One may take the inverse limit of the spaces $X_{E}$ over all entourages
with bonding maps $\phi _{EF}:X_{F}\rightarrow X_{E}$ for $F\subset E$
defined by $\phi _{EF}([\alpha ]_{F})=[\alpha ]_{E}$ to obtain a space $%
\widetilde{X}$ and projection map $\phi :\widetilde{X}\rightarrow X$ (here
we identify $X$ with $X_{X\times X}$). In \cite{BPUU} this system was
referred to as the \textit{fundamental inverse system}. The inverse limit of
the groups $\pi _{E}(X)$ is called the \textit{uniform fundamental group} $%
\pi _{U}(X)$. When $X$ is \textit{weakly chained }in the sense of \cite{PWC}
(which is \textit{a priori} stronger than the notion of \textquotedblleft
coverable\textquotedblright\ in \cite{BPUU}), the projection $\phi $ is a
generalized regular covering map with deck group $\pi _{U}(X)$, called the 
\textit{uniform universal cover}. In \cite{PWC} we extended the results of 
\cite{BPUU} for weakly chained $X$, showing that $\widetilde{X}$ is
uniformly simply connected in the sense that $\pi _{U}(X)$ is trivial, and
is unique up to uniform homeomorphism. However, some continua are not weakly
chained, and for these spaces the projection $\phi $ need not be surjective.
For example, for a standard solenoid, $\phi $ maps only onto the path
component. Throughout this construction, given a basepoint $\ast $ in $X$,
we may choose $[\ast ]_{E}$ for the basepoint in $X_{E}$ and $([\ast ]_{E})$
for the basepoint in $\widetilde{X}$.

\begin{remark}
\label{close}The idea that \textquotedblleft $E$-close $F$-chains are $EF$%
-homotopic\textquotedblright , has been used in one form or another in
several papers involving discrete homotopy theory. This concept is
particularly important for counting arguments in compact spaces, and in this
paper we will use it in the proof of Theorem \ref{factor}. For completeness
and as an illustration for the unfamiliar reader, we will sketch the
argument that $E$-close $E$-chains are $E^{2}$-homotopic. Suppose that $%
\alpha =\{x=x_{0},x_{1},...,x_{n-1},x_{n}=y\}$ and $\beta =\{x=x_{0}^{\prime
},x_{1}^{\prime },...,x_{n-1}^{\prime },x_{n}^{\prime }=y\}$ are $E$-chains
from $x$ to $y$ such that for all $i$, $(x_{i},x_{i}^{\prime })\in E$. As
the reader can readily verify, the following are $E^{2}$-homotopy basic
moves, meaning that at each stage we always have an $E^{2}$-chain. For
example, since $(x_{0},x_{1})\in E$ and $(x_{1},x_{1}^{\prime })\in E$, $%
(x_{0},x_{1}^{\prime })\in E^{2}$. 
\begin{equation*}
\alpha \rightarrow \{x_{0},x_{1}^{\prime
},x_{1},x_{2},...,x_{n}\}\rightarrow \{x_{0},x_{1}^{\prime
},x_{2},...,x_{n}\}
\end{equation*}%
\begin{equation*}
\rightarrow \{x_{0},x_{1}^{\prime },x_{2}^{\prime
},x_{2},x_{3},...,x_{n}\}\rightarrow \{x_{0},x_{1}^{\prime },x_{2}^{\prime
},x_{3},...,x_{n}\}\rightarrow \cdot \cdot \cdot \rightarrow \beta
\end{equation*}%
There is no loss of generality in assuming that $\alpha $ and $\beta $ have
the same number of points because one can always bring this about by
duplicating points, which does not change the $E$-homotopy class. There are
many variations on this theme, e.g. \textquotedblleft $E^{n}$-close $E^{m}$%
-chains are $E^{n+m}$-homotopic\textquotedblright . Even if a second chain
is not already known to be an $E$-chain, if it is $E$-close to an $E$-chain
then it must be an $E^{3}$-chain that is $E^{3}$-homotopic to the other. In
the proof of Theorem \ref{corr}, which implies Theorem \ref{factor}, there
are two stages of loss and we begin with $F^{6}\subset E$ because the first
stage results in $F^{3}$-chains, which converge to an $F^{4}$-chain, and
therefore when they are $F$-close to the limit they are $F^{6}$-homotopic,
hence $E$-homotopic to the limit. A constraint in such arguments is the
question of \textquotedblleft refinability\textquotedblright , i.e. when is
there an $F$-chain in the $E$-homotopy class of an $E$-chain? In the case of
Theorem \ref{corr} we start with loops $[\lambda ]_{E}$ in $\theta _{EF}(\pi
_{F}(X))$, which by definition means that we may take $\lambda $ to be an $F$%
-chain.
\end{remark}

Given a uniform space $X$ and entourage $E$, a subset $A$ is called $E$%
\textit{-chain connected} if every pair of points in $A$ is joined by an $E$%
-chain in $A$. $A$ is called \textit{chain connected }(equivalent to
\textquotedblleft uniformly connected\textquotedblright\ in the classical
literature) if $A$ is $E$-chain connected for every entourage $E$. In this
context we often instead use phrases like \textquotedblleft $x$ and $y$ are
joined by arbitrarily fine chains in $A$\textquotedblright . In the compact
case with the unique compatible uniformity, connected and chain connected
are equivalent. While we will not formally use it in this paper, we will
define \textquotedblleft weakly chained\textquotedblright\ (\cite{PWC})
because the definition is very simple and weakly chained continuua are a
kind of intermediate space between Poincar\'{e} spaces and arbitrary
continua, which makes them helpful in the exposition of the current paper. A
uniform space $X$ is\textit{\ weakly chained }if $X$ is chain connected and
for every entourage $E$ there exists an entourage $F$ such that if $(x,y)\in
F$, $x,y$ may be joined by arbitrarily fine chains $\alpha $ such that $%
[\alpha ]_{E}=[x,y]_{E}$. Weakly chained spaces need not be locally path
connected, but as mentioned above, solenoids are not weakly chained.

If $x\in X$, the $E$\textit{-chain component} $U_{x}^{E}$ of $x$ in $X$ is
the union of all $E$-chain connected sets in $X$ containing $x$, which is
clearly itself $E$-chain connected. Moreover, $U_{x}^{E}$ is also the
intersection of all sets containing $x$ that are \textit{uniformly }$E$%
\textit{-open} in the the sense that if $z\in U_{x}^{E}$ then $B(z,E)\subset
U_{x}^{E}$, see Lemma 9, \cite{PWC}. Combining these two equivalences we
have the following lemma that we will use without reference.

\begin{lemma}
Suppose that $X$ is a uniform space and $x\in A\subset X$. If $E$ is an
entourage in $X$ then $A=U_{x}^{E}$ if and only if $A$ is both $E$-chain
connected and uniformly $E$-open.
\end{lemma}

We say that $U$ is \textit{uniformly open} if $U$ is uniformly $E$-open for
some $E$. Note that the complement of any uniformly $E$-open set is also
uniformly $E$-open, and in particular uniformly open sets are also closed.
It follows that $X$ is chain connected if and only if $X$ is the only
non-empty uniformly open subset of $X$. The \textit{chain component} of $%
x\in X$ is defined to be the intersection of all $U_{x}^{E}$; equivalently
the chain component of $x$ is the union of all chain connected sets
containing $x$.

\begin{remark}
We note for once and for all that when $X$ is chain connected, the
construction of $\phi _{E}:X_{E}\rightarrow X$ is independent, up to
equivariant homeomorphism, of choice of basepoint in $X$ (\cite{BPUU},
Remark 18). The proof of this fact is similar to the proof from traditional
covering space theory.
\end{remark}

A summary of known connectivity properties of $X_{E}$ is as follows. If $X$
is \textit{coverable} in the sense of \cite{BPUU}, which includes if $X$ is
a uniform Poincar\'{e} space, then equivalently there is a basis consisting
of entourages $E$ such that $X_{E}$ is chain connected (this is automatic
if, for example, $E$ has chain connected balls). If $X$ is weakly chained
then for any entourage, the chain components of $X_{E}$ are uniformly open.
Weakly chained is \textit{a priori} weaker than coverable and it is
presently unknown whether these properties are equivalent in the metrizable
case. If $X$ is a solenoid then the chain components are generally not
uniformly open because, as with the solenoid itself, $X_{E}$ is locally a
product of a Cantor set and an interval, and $\pi _{E}(X)$ acts by
\textquotedblleft permuting\textquotedblright\ some of the path components
(see Example \ref{solenoid}).

\section{Quotients and Covering Maps}

We recall the basics of quotients of uniform spaces worked out in \cite{PQ}.
Let $X$ be a uniform space and $G$ be a group of bijections of $X$. $G$ is
said to act \textit{isomorphically} on $X$ if $X$ has a basis of entourages $%
E$ that are invariant with respect to $G$; that is $g(E)=E$ for all $g\in G$%
. Quotients may be considered in the more general setting of
\textquotedblleft equiuniform\textquotedblright\ actions (\cite{PQ}), but we
do not need such generality here. The quotient space $X/G$ is defined to be
the set of all orbits $Gx:=\{g(x):g\in G\}$; the map $\pi :X\rightarrow X/G$
defined by $\pi (x)=Gx$ is called the quotient map. From the definition of
isomorphic action it is clear that the maps in $G$ are all uniform
homeomorphisms of $X$. According to Theorem 11, \cite{PQ}, for an isomorphic
action the set of all $\pi (E)$ such that $E$ is an entourage in $X$ is a
uniform structure on $X/G$ that is compatible with the quotient topology.
This is called the \textit{quotient uniformity}, and will be the default
uniformity on $X/G$. Then the quotient map $\pi :X\rightarrow X/G$ is
bi-uniformly continuous.

Metrizability questions are generally simpler for uniform spaces, which are
metrizable if and only if they are Hausdorff and have a countable basis. In
particular, if the quotient of an isomorphic action on a metrizable uniform
space is Hausdorff, it is metrizable. In what follows, we will use these
facts without reference.

\begin{remark}
\label{quotient}If $f:X\rightarrow Y$ is a bi-uniformly continuous mapping
and $G$ acts isomorphically on $X$ such that for all $x\in X$, $%
f^{-1}(f(x))=Gx$ then the mapping $f(x)\leftrightarrow Gx$ is a uniform
homeomorphism and we may identify $f$ as the quotient map $f:X\rightarrow
Y=X/G$. We will use this fact frequently to show that a bi-uniformly
continuous function $f$ is in fact a quotient map. Our shorthand for this
statement will be \textquotedblleft the point preimages of $f$ are the
orbits of $G$\textquotedblright .
\end{remark}

\begin{remark}
We note that as was so colorfully stated in \cite{I}, the idea that general
topological quotients of uniform spaces have a compatible uniformity is
\textquotedblleft horribly false\textquotedblright . But quotients via
suitable group actions provide the uniformness of identifications that is
needed to make it work. Under the right circumstances we also may work in
reverse as exemplified in the definition of $F^{\ast }$ in the previous
section: start with a uniform structure on the quotient space and lift to a
uniform structure on the original. The following definition was not stated
in \cite{PQ} but it will be important in the current paper.
\end{remark}

\begin{definition}
\label{equivdef}Suppose that $G_{i}$ acts isomorphically on a uniform space $%
Y_{i}$ for $i=1,2$. We say the actions are equivalent if there are a uniform
homeomorphism $h:Y_{1}\rightarrow Y_{2}$ and an isomorphism $\theta
:G_{1}\rightarrow G_{2}$ that are compatible in the following sense: for all 
$g\in G_{1}$, $h\circ g=\theta (g)\circ h$. The pair $(h,\theta )$ is called
an equivalence of the actions. We say that two quotient maps $\pi
_{i}:Y_{i}\rightarrow Y_{i}/G_{i}$ are equivalent if the actions of $G_{i}$
on $Y_{i}$ are equivalent.
\end{definition}

Note that the equivalence described above is indeed an equivalence relation.
For example, if $(h,\theta )$ is an equivalence of the actions of $G_{1}$ on 
$Y_{1}$ and of $G_{2}$ on $Y_{2}$ then $(h^{-1},\theta ^{-1})$ is an
equivalence of actions of $G_{2}$ on $Y_{2}$ and $G_{1}$ on $Y_{1}$. In
fact, for any $k=\theta (g)$ in $G_{2}$ and $y=h(x)$ in $Y_{2}$, 
\begin{equation*}
h^{-1}(k(y))=\theta ^{-1}(k)(h^{-1}(y))\Leftrightarrow \theta
(g)(h(x))=h(g(x))\text{.}
\end{equation*}

\begin{example}
\label{strong}Consider the compact space $Y=G\times \lbrack 0,1]$, where $G$
is the group with three elements with the discrete topology, with $G$ acting
on $Y$ by cycling the components, with quotient $[0,1]$. That is, for $z\in
\{0,1,2\}$, $g_{\overline{z}}((\overline{w},x))=(\overline{z+w},x)$.
Consider the homeomorphism $h$ that fixes the component $\{\overline{0}%
\}\times \lbrack 0,1]$ and exchanges the other two components, i.e. $h((%
\overline{z},x))=(\overline{2z},x)$. Then $h$ is a covering equivalence in
the traditional sense, but to be an equivalence in the current sense would
require $\theta :G\rightarrow G$ to be the identity map. But for example $%
h(g_{\overline{1}}((\overline{1},x)))=h((\overline{2},x))=(\overline{1},x)$,
while $\theta (g_{\overline{1}})(h((\overline{1},x))=g_{\overline{1}}((%
\overline{2},x))=(\overline{0},x)$. So compatibility fails. The next lemma
shows that this notion of equivalence is in general stronger than the
traditional one, and we will see (Corollary \ref{eqeq}) that the two notions
are equivalent for generalized regular covering maps.
\end{example}

\begin{lemma}
\label{induced}Suppose that $(h,\theta )$ is an equivalence of isomorphic
actions of $G_{i}$ on $Y_{i}$ for $i=1,2$. Then the mapping $\overline{h}%
(G_{1}y):=G_{2}h(y)$ is a uniform homeomorphism from $Y_{1}/G_{1}$ to $%
Y_{2}/G_{2}$ such that if $\pi _{i}:Y_{i}\rightarrow Y_{i}/G_{i}$ is the
quotient map, then $\overline{h}\circ \pi _{1}=\pi _{2}\circ h$.
\end{lemma}

\begin{proof}
We first show that $\overline{h}$ is well-defined. If $z\in G_{1}y$ then $%
z=g(y)$ for some $g\in G_{1}$. By definition, $h(z)=\theta (g)(h(y))$ and
therefore $h(z)\in G_{2}y$ and $G_{2}z=G_{2}y$. By symmetry the proof that $%
\overline{h}$ is a bijection will be complete if we show that $\overline{%
h^{-1}}\circ \overline{h}$ is the identity. But in fact 
\begin{equation*}
\overline{h^{-1}}(\overline{h}(G_{1}y))=\overline{h^{-1}}%
(G_{2}h(y))=G_{1}h^{-1}(h(y))=G_{1}y\text{. }
\end{equation*}%
Finally, the fact that $\overline{h}$ is a uniform homeomorphism follows
from the definition of the quotient uniformity and the fact that $h$ is a
uniform homeomorphism.
\end{proof}

For any entourage $E$ in a uniform space $X$ on which $G$ acts
isomorphically, let 
\begin{equation*}
U_{E}(G):=\{g\in G:(x,g(x))\in E\text{ for all }x\in X\}
\end{equation*}%
\begin{equation*}
S_{E}(G):=\{g\in G:(x,g(x))\in E\text{ for some }x\in X\}
\end{equation*}%
and $N_{E}(G)$ denote the subgroup generated by $S_{E}(G)$ in $G$. According
to Theorem 31 in \cite{PQ}, since the action of $G$ is isomorphic, $N_{E}(G)$
is a normal subgroup of $G$. If for some entourage $F$, $N_{F}(G)=1$, then $%
G $ is said to act \textit{discretely}. We say that $G$ acts \textit{%
prodiscretely} if for every entourage $E$ there is some entourage $F$ such
that $N_{F}(G)\subset U_{E}(G)$. The collection of all $U_{E}(G)$ where $E$
is an entourage in $X$ forms a basis at the identity for the topology of
uniform convergence in $G$, making it a topological group (Theorem 17, \cite%
{PQ}). We note that when $X$ is metrizable and therefore has a countable
base for its uniform structure, $G$ has a countable base of sets $U_{E}(G)$
and therefore is metrizable. This topology gives rise, as described in the
second section, to the left and right uniformities of uniform convergence.

\begin{remark}
There are some subtleties related to invariance properties of the uniform
structure and topology on a group $G$ acting on a uniform space--for more
information see the exact statement of Theorem 17 and Notation 18 in \cite%
{PQ}. These details are unimportant for this paper due to Lemma 19 in \cite%
{PQ}, which implies that in the case of isomorphic actions, the left and
right uniformities of uniform convergence of $G$ are the same.
\end{remark}

We have the following basic results (all references in this paragraph are
from \cite{PQ}): If $X$ is Hausdorff, so is $G$, and if $X$ is complete (as
a uniform space, defined analogously to completeness in a metric space) and $%
G$ is a closed subset of the group of all uniform homeomorphisms of $X$ then 
$G$ is complete (Theorem 20). In the case of a prodiscrete action we have
the following: The set of all $N_{E}(G)$, where $E$ is an entourage of $X$,
is a basis at $1$ for the topology of $G$ (Corollary 32); the action of $X$
is free (Lemma 33); and if $G$ is complete then $G$ is a prodiscrete
topological group (i.e. topologically the inverse limit of discrete groups)
(Corollary 34). If $G$ acts discretely then $G$ is discrete, hence complete,
each of its orbits is uniformly discrete, and the action is properly
discontinuous (Proposition 22). In fact, a discrete action is essentially a
\textquotedblleft uniformly properly discontinuous\textquotedblright\
action. Here \textquotedblleft $A$ is uniformly discrete\textquotedblright\
means that for some entourage $E$, every $E$-ball in the uniform space
contains at most one point of $A$.

Suppose $G$ acts isomorphically on $X$ and $H$ is a subgroup of $G$. Then
clearly $H$ also acts isomorphically on $X$, and the action of $H$ is
discrete (resp. prodiscrete) if the action of $G$ is. When $H$ is a normal
subgroup of $G$ then $G/H$ acts isomorphically on $X/H$ (with the quotient
uniformity) and we have the induced quotient mapping $\phi :X/H\rightarrow
\left( X/H\right) /\left( G/H\right) =X/G$ (Proposition 28, \cite{PQ}). The
latter equality is via the uniform homeomorphism $(G/H)(Hx)\leftrightarrow
Gx $. Although we proved the following statement for prodiscrete actions in
Theorem 31, \cite{PQ}, we did not prove it for discrete actions, so we will
do so now:

\begin{lemma}
\label{addon}If $G$ acts discretely and isomorphically on a uniform space $X$
and $H$ is a normal subgroup of $G$ then $G/H$ acts discretely (and
isomorphically) on $X/H$.
\end{lemma}

\begin{proof}
Let $\eta :X\rightarrow X/H$ be the quotient map and let $E$ be an invariant
entourage such that $N_{E}(G)$ is trivial. We claim that $N_{\eta (E)}(G/H)$
is also trivial, for which it suffices to prove that $S_{\eta (E)}(G/H)$
contains only the trivial element. If $gH\in S_{\eta (E)}(G/H)$ then by
definition $\eta (E)$ contains $(Hx,gH(Hx))=(Hx,Hg(x))$ for some $x\in X$.
This means that for some $h_{1},h_{2}\in H$, 
\begin{equation*}
(h_{1}(x),h_{2}(g(x)))\in E\text{.}
\end{equation*}%
Since $E$ is invariant, this means $(x,h_{1}^{-1}(h_{2}(g(x)))\in E$. That
is, $h_{1}^{-1}h_{2}g\in N_{E}(G)$ and therefore is trivial. Therefore $g\in
H$ and $Hg$ is trivial.
\end{proof}

\begin{proposition}
\label{haus}If $G$ is a complete group that acts prodiscretely and
isomorphically on a metrizable uniform space $X$ then the orbits of $G$ are
closed and $X/G$ is metrizable.
\end{proposition}

\begin{proof}
Let $f:X\rightarrow X/G$ be the quotient map. Let $\{E_{i}\}$ be a countable
nested invariant basis for $X$ and recall that the set of all $N_{E_{i}}(G)$
is a basis at the identity for the topology of uniform convergence on $G$,
which means that for every $i$ there is some $j$ such that 
\begin{equation}
N_{E_{j}}(G)\subset U_{E_{i}}(G)\text{.}  \label{need}
\end{equation}%
This implies the following: if $g\in G$ has the property that $%
(g(x_{0}),x_{0})\in E_{j}$ for some $x_{0}\in X$ then $(g(x),x))\in E_{i}$
for all $x\in X$. Now suppose that $y$ is in the closure of some orbit $Gx$.
That is, for every entourage $E_{j}$, there is some $x_{j}:=g_{j}(x)\in
B(y,E_{j})$. But now Inclusion (\ref{need}) implies that $\{g_{j}\}$ is
Cauchy in the topology of uniform convergence. Since $G$ is complete, $%
g_{j}\rightarrow g\in G$. This means that for any invariant entourage $E$, $%
(g(x),g_{i}(x))\in E$ for all large $i$. In other words, $g(x)=y$, so $y\in
Gx$.

To finish the proof, we need only show that $X/G$ is Hausdorff. Let $x\in X$
and suppose that $y$ is not in $Gx$. We first note that there is some
invariant entourage $E$ such that $y\notin B(w,E)$ for every $w\in Gx$. If
this were not true, then for every $i$ there would be some $w_{i}\in G_{x}$
such that $y\in B(w_{i},E_{i})$, which is equivalent to $w_{i}\in B(y,E_{i})$%
. But this implies that $y$ is in the closure of $G_{x}$, so $y\in Gx$, a
contradiction.

Suppose $F$ is an invariant entourage such that $F^{2}\subset E$. Let $%
U:=\dbigcup\limits_{u\in Gy}B(u,F)$ and $W:=\dbigcup\limits_{w\in Gx}B(w,F)$%
. We claim that $U\cap W=\varnothing $. If $z\in U\cap W$ then $z\in
B(u,F)\cap B(w,F)$ for some $u\in Gy$ and $w\in Gx$. This means that $%
(u,w)\in F^{2}\subset E$. Since $E$ is invariant, if $g(u)=y$, then $%
(y,g(w))\in E$. Since $g(w)\in Gx$, this contradicts our choice of $E$.

To finish the proof, we claim that $f(U)\cap f(W)=\varnothing $. Suppose
there is some $a\in f(U)\cap f(W)$. So there exist $a^{\prime }\in U$ and $%
a^{\prime \prime }\in W$ such that $f(a^{\prime })=f(a^{\prime \prime })=a$.
This last equation implies that there is some $k\in G$ such that $%
k(a^{\prime })=a^{\prime \prime }$. By definition of $U$ and $W$, we have
that $(a^{\prime },g_{1}(y)),(a^{\prime \prime },g_{2}(x))\in F$ for some $%
g_{1},g_{2}\in G$. Since $F$ is invariant, this means that $%
(k(g_{1}(y)),k(a^{\prime }))=(k(g_{1}(y)),a^{\prime \prime })\in F$. Since $%
k(g_{1}(y))$ is in the orbit of $Y$, this means that $a^{\prime \prime }\in
U\cap W$, a contradiction.
\end{proof}

\begin{proposition}
\label{inducedp}Let $h:Z\rightarrow X=Z/G$ be a quotient via an isomorphic
action by $G$, and $g:Z\rightarrow Y$ and $f:Y\rightarrow X$ be bi-uniformly
continuous functions such that $h=f\circ g$. Suppose that for some normal
subgroup $H$ of $G$, the point preimages of $g$ are the orbits of $H$. Then $%
f$ is a quotient map equivalent to the induced quotient map $\pi
:Z/H\rightarrow (Z/H)/(G/H)$.
\end{proposition}

\begin{proof}
By assumption, $g$ is the quotient map $g:Z\rightarrow Y=Z/H$. We know
already that $G/H$ acts isomorphically on $Y=Z/H$ via the action $%
kH(g(z))=g(k(z))$. Therefore we need only show that the point preimages of $%
f $ are the orbits of $G/H$. In fact, $f(g(x))=f(g(y))\Leftrightarrow
h(x)=h(y) $, which is equivalent to $k(x)=y$ for some $k\in G$, which in
turn is equivalent to $kH(g(x))=g(y)$.
\end{proof}

\begin{definition}
As usual, if $G$ is a group of bijections of a set $Y$ and $C\subset Y$, the 
\textit{stabilizer subgroup} of $C$ is the subgroup $S_{C}$ of $G$
consisting of all $g\in G$ such that $g(C)=C$. We will say that $C$ has 
\textit{Property ST} if whenever $g\in G$, if $g(C)\cap C\neq \varnothing $
then $g\in S_{C}$.
\end{definition}

\begin{lemma}
\label{cc}Suppose that a group $G$ acts isomorphically on a uniform space $Y$%
.

\begin{enumerate}
\item If $C$ is any chain component of $Y$ (resp. $E$-chain component of $Y$
for some invariant entourage $E$) then $C$ has Property ST.

\item If $C$ is uniformly $E$-open for some entourage $E$ and the
restriction $f_{C}$ of the quotient map $f:Y\rightarrow X=Y/G$ to $C$ is
surjective then:

\begin{enumerate}
\item $f_{C}$ is bi-uniformly continuous.

\item If $C$ has property ST then $f_{C}:C\rightarrow X=Y/S_{C}$ is a
quotient map.
\end{enumerate}
\end{enumerate}
\end{lemma}

\begin{proof}
For the first part, suppose that for some $x$ and $g\in G$, $g(x)\in C$. For
any $z\in C$ there is an arbitrarily fine chain (resp. an $E$-chain) $\alpha 
$ from $x$ to $z$. Then $g(\alpha )$ is an arbitrarily fine chain (resp. $E$%
-chain) in $Y$ from $g(x)$ to $g(z)$, so $g(z)\in C$, i.e. $g(C)\subset C$.
Since $g$ is a bijection, for every $y\in C$ there is some $w\in Y$ such
that $g(w)=y$. But by what we just showed, $w$ cannot lie in a different
component (resp. $E$-chain component) of $Y$. Therefore, $C\subset g(C)$.

For Part 2a, note that by elementary results $f_{C}$ is uniformly continuous
with respect to the subspace uniformity. Let $F\subset E$ be an invariant
entourage, and note that $C$ is also uniformly $F$-open. We will show that $%
f(F)\subset f(F\cap \left( C\times C\right) )$, showing that the latter is
an entourage, proving the claim. Let $(x,y)\in f(F)$. So there exists $%
(x^{\prime },y^{\prime })\in F$ such that $(f(x^{\prime }),f(y^{\prime
}))=(x,y)$. Since $f_{C}$ is surjective there is some $x^{\prime \prime }\in
C$ such that $f(x^{\prime \prime })=x$. Therefore there is some $g\in G$
such that $g(x^{\prime })=x^{\prime \prime }$. Since $F$ is invariant, $%
(x^{\prime \prime },g(y^{\prime }))\in F$. Since $C$ is uniformly $F$-open, $%
g(y^{\prime })\in C$. That is, $(x^{\prime \prime },g(y^{\prime }))\in F\cap
\left( C\times C\right) $ and since $f((x^{\prime \prime },g(y^{\prime
}))=(x,y)$, $(x,y)\in f(F\cap \left( C\times C\right) )$.

For Part 2b we need only show that the orbits of $S_{C}$ are the point
preimages of $f_{C}$. Clearly if $z=g(w)$ for some $z,w\in C$ and $g\in
S_{C}\subset G$, $f(z)=f(w)$. Conversely, suppose the latter equation holds
for $z,w\in C$. Then there is some $g\in G$ such that $g(z)=w$. Since $%
z,w\in C$ and $C$ has property ST, $g\in S_{C}$.
\end{proof}

\section{Uniformizing Regular Covering Maps}

There is an unfortunate variety of terminologies involving actions on
topological spaces by discrete groups, see for example \cite{Kap} for a
discussion. We will say that a group $G$ of homeomorphisms of a metrizable
topological space $Y$ \textit{acts properly discontinuously} if (1) for
every $y\in Y$ there is some open set $U$ such that if $g\in G$ and $%
g(U)\cap U\neq \varnothing $ then $g=1$; and (2) the orbit space $X=Y/G$
with the quotient topology is Hausdorff.

For spaces that are not Poincar\'{e} spaces, the notion of \textquotedblleft
regular covering map\textquotedblright\ is problematic since this is
typically defined in terms of the traditional mapping from the group of deck
transformations into the fundamental group, which requires paths. The
equivalent notion, that the deck (covering transformation) group is
\textquotedblleft transitive on the fibers\textquotedblright , simply means
that the map is a quotient map. Therefore we will take our definition of a 
\textit{regular covering map} between topological spaces to be that the map
is a covering map in the usual sense and also a quotient map via a free
action by group of homeomorphisms. Recall that a covering map $%
f:X\rightarrow Y$ by definition has the property that every $x\in X$ is
contained in an \textit{evenly covered} open set $U$, i.e. $f^{-1}(U)$ is a
disjoint collection of open sets, the restriction to any of which is a
homeomorphism onto $U$.

\begin{definition}
Suppose $G$ is a group of bijections on a set $Y$. The PD-domain $P$ of the
action is the set of all $(x,y)\in Y\times Y$ such that if $y=g(x)$ for some 
$g\in G$ then $x=y$. If $R$ is a symmetric, invariant open set containing
the diagonal in $Y\times Y$ such that $R^{2}\subset P$ then $R$ is called a
root domain.
\end{definition}

Put another way, the PD-domain consists of all ordered pairs $(y,z)$ such
that if $y\neq z$, $z$ does not lie in the orbit of $y$. Clearly the
PD-domain is a symmetric, invariant subset of $Y\times Y$ containing the
diagonal. If the action is free then $x=y$ is strengthened to $g=1$.

If $R$ is a root domain then of course any symmetric open set containing the
diagonal that is contained in $R$ is again a root domain.

\begin{example}
Consider the action of $\mathbb{Z}$ on $\mathbb{R}$ generated by $x\mapsto
x+1$. As can easily be checked, the PD-domain $P$ of this action consists of
the plane, removing all lines that are shifts of the diagonal line
vertically by an integer different from $0$. But it is simpler to look at
the $P$-balls, which determine $P$. $B(0,P)$ consists of $\mathbb{R}$ with
all non-zero integers removed. For simplicity, give $\mathbb{R}$ its
standard metric and suppose that $R=E_{\varepsilon }$ is a root domain for
some $\varepsilon >0$. Since the metric is geodesic, $B(x,R^{2})=B(x,2%
\varepsilon )$, and we must have $B(0,2\varepsilon )\subset B(0,P)$, which
forces $\varepsilon \leq \frac{1}{2}$. That is, even though $P$ is in a
sense very large (it is dense in $\mathbb{R}^{2}$), the requirement that $%
R^{2}\subset P$ forces $R$ to be relatively small.
\end{example}

\begin{definition}
Suppose $G$ acts properly discontinuously on a metrizable topological space $%
Y$ with quotient map $f:Y\rightarrow X=Y/G$. A symmetric open set $E$
containing the diagonal in $X\times X$ is called evenly covered with respect
to a root domain $R$ if $E\subset f(R)$. If $E$ is evenly covered with
respect to some $R$, we simply say that $E$ is evenly covered.
\end{definition}

If $R_{1}\subset R_{2}$ are root domains and $E$ is evenly covered with
respect to $R_{1}$ and $R_{2}$ then by elementary set theory 
\begin{equation}
E_{R_{1}}^{\ast }=E_{R_{2}}^{\ast }\cap R_{1}\subset E_{R_{2}}^{\ast }\text{.%
}  \label{nest}
\end{equation}

\begin{proposition}
\label{domain}Let $G$ be a group of homeomorphisms on a metrizable
topological space $Y$ that acts properly discontinuously with quotient map $%
f:Y\rightarrow X=Y/G$. Let $P$ be the PD-domain of the action. Then

\begin{enumerate}
\item If $(x,y)\in P$ and $x\neq y$ then $f(x)\neq f(y)$.

\item If $E$ is a symmetric open set containing the diagonal in $X$ and $R$
is a root domain in $Y$, then $E_{R}^{\ast }:=f^{-1}(E)\cap R$ is a
symmetric, invariant open set containing the diagonal in $Y$, called the
lift of $E$ with respect to $R$. Moreover, if $E$ is evenly covered with
respect to $R$ then $f(E_{R}^{\ast })=E$.
\end{enumerate}
\end{proposition}

\begin{proof}
If $(x,y)\in P$ and $f(x)=f(y)$ then there is some $g\in G$ such that $%
g(x)=y $, which is impossible by definition of $P$ unless $x=y$.

For the second part, note that if $W$ is any subset of $X\times X$ then $%
f^{-1}(W)$ is invariant with respect to the action of $G$. In fact, if $%
(x,y)\in f^{-1}(W)$ then $(f(x),f(y))\in W$. But for any $g\in G$, $%
f(g(x))=f(x)$ and $f(g(y))=f(y)$. Since $(f(x),f(y))\in W$, $(g(x),g(y))\in
f^{-1}(W)$. Invariance of $E_{R}^{\ast }$ now follows from the fact that $R$
is invariant. From elementary set theory, $f(E_{R}^{\ast })=f(f^{-1}(E)\cap
R)\subset E\cap f(R)\subset E$. Finally, suppose that $(x,y)\in E\subset
f(R) $; that is, there is some $(x^{\prime },y^{\prime })\in R$ such that $%
f(x^{\prime })=x$ and $f(y^{\prime })=y$. But since $(x,y)\in E$, $%
(x^{\prime },y^{\prime })\in f^{-1}(E)$ and hence $(x^{\prime },y^{\prime
})\in E_{R}^{\ast }$. That is, $(x,y)\in f(E_{R}^{\ast })$.
\end{proof}

\begin{proposition}
\label{evenly}Let $G$ be a group of homeomorphisms on a metrizable
topological space $Y$ that acts properly discontinuously with quotient map $%
f:Y\rightarrow X=Y/G$. If $E$ is evenly covered in $X$ with respect to some
root domain $R$, then $E$-balls are evenly covered by $E_{R}^{\ast }$-balls.
\end{proposition}

\begin{proof}
Suppose $y,y^{\prime }\in f^{-1}(x)$ and there is some $w\in B(y,E_{R}^{\ast
})\cap B(y^{\prime },E_{R}^{\ast })$. Then $(y,y^{\prime })\in \left(
E_{R}^{\ast }\right) ^{2}=(f^{-1}(E)\cap R)^{2}\subset R^{2}\subset P$,
where $P$ is the PD-domain. Since $f(y)=x=f(y^{\prime })$, by Proposition %
\ref{domain}, $y=y^{\prime }$. Now let $h$ be the restriction of $f$ to some 
$B(y,E_{R}^{\ast })$ with $y\in f^{-1}(x)$. By Proposition \ref{domain}, $%
f(B(y,E_{R}^{\ast }))\subset B(x,E)$. Let $w\in B(x,E)$. By Proposition \ref%
{domain} there exist $\left( y^{\prime },z^{\prime }\right) \in E_{R}^{\ast
} $ with $f(y^{\prime })=x$ and $f(z^{\prime })=w$. Since $f(y^{\prime
})=x=f(y)$ there is some $g\in G$ such that $g(y^{\prime })=y$; letting $%
z:=g(z^{\prime })$ we have $f(z)=f(g(z^{\prime }))=f(z^{\prime })=w$. Since $%
E_{R}^{\ast }$ is invariant, $z\in B(y,E_{R}^{\ast })$, showing that $h$ is
onto $B(x,E)$. Finally, suppose that $z_{1},z_{2}\in B(y,E_{R}^{\ast })$ and 
$f(z_{1})=f(z_{2})$. Then $(z_{1},z_{2})\in (E_{R}^{\ast })^{2}\subset
R^{2}\subset P$ and by Proposition \ref{domain}, $z_{1}=z_{2}$.
\end{proof}

We will use the next proposition frequently without reference.

\begin{proposition}
Suppose $f:Y\rightarrow X=Y/G$ is a quotient map between metric spaces. Then 
$f$ is a regular covering map if and only if $G$ acts properly
discontinuously and has a root domain.
\end{proposition}

\begin{proof}
Suppose that $f$ is a regular covering. Since $Y$ is assumed to be metric,
hence Hausdorff, we need only verify the first condition in the definition
of proper discontinuous action. Let $y\in Y$. Then there is some evenly
covered open set $V$ containing $f(y)$. Let $U$ be an open subset of $%
f^{-1}(V)$ containing $y$ such that $U$ contains $x$ and the restriction of $%
f$ to $U$ is 1-1. If $z\in g(U)\cap U$ then $z$ and $g^{-1}(z)$ both lie in $%
U$ and by choice of $U$, $z=g^{-1}(z)$ and $g=1$.

For every $x\in X$ there is some maximal $\varepsilon _{x}>0$ such that $%
B(x,\varepsilon _{x})$ is evenly covered. By maximality and the triangle
inequality, for any $x_{1},x_{2}\in X$, 
\begin{equation}
\varepsilon _{x_{1}}\geq \varepsilon _{x_{2}}-d(x_{1},x_{2})\text{.}
\label{teqf}
\end{equation}

For $y\in Y$ there is a unique open set $U_{y}$ containing $y$ such that the
restriction of $f$ to $U_{y}$ is a homeomorphism onto $B(f(y),\frac{%
\varepsilon _{f(y)}}{3})$. Define $R:=\{(w,z):w,z\in U_{y}$ for some $y\in
Y\}$, which is clearly a symmetric open set containing the diagonal. Now
suppose that $g\in G$. Then $f(g(U_{y}))=f(U_{y})=B(f(y),\frac{\varepsilon
_{f(y)}}{3})$ and since the restriction of $f$ to $U_{y}$ is a bijection
onto $B(f(y),\frac{\varepsilon _{f(y)}}{3})$, so is the restriction of $f$
to $g(U_{y})$. That is, $g(U_{y})=U_{g(y)}$ and $R$ is invariant.

Now suppose that $(w_{1},w_{2})\in R^{2}$. That is, for some $z$, $%
(w_{1},z),(w_{2},z)\in R$. This in turn means that for some $y_{1},y_{2}$, $%
(w_{i},z)\in U_{y_{i}}$, for $i=1,2$. Therefore $f(w_{i}),f(z)\in B(y_{i},%
\frac{\varepsilon _{y_{i}}}{3})$. From Formula (\ref{teqf}) we have $%
\varepsilon _{f(z)}\geq \varepsilon _{f(y_{i})}-d(f(z),f(y_{i}))>\varepsilon
_{f(y_{i})}-\frac{\varepsilon _{f(y_{i})}}{3}=\frac{2\varepsilon _{f(y_{i})}%
}{3}$. We also have $d(f(w_{i}),f(z))<\frac{2\varepsilon _{f(y_{i})}}{3}$
and therefore $f(w_{i})\in B(f(z),\varepsilon _{f(z)})$. Since the latter
ball is evenly covered, there is an open set $V$ containing $z$ such that
the restriction of $f$ to $V$ is 1-1. Since $U_{y_{i}}$ contains $z$ and
maps into $B(f(z),\varepsilon _{f(z)})$ it must be contained in $V$.
Therefore both $w_{1}$ and $w_{2}$ lie in $V$ and it is impossible for a
non-trivial $g\in G$ to have the property that $g(w_{1})=w_{2}$. That is $%
R^{2}$ is contained in the PD-domain and by definition $R$ is a root domain.

The converse is an immediate consequence of Proposition \ref{evenly}.
\end{proof}

\begin{theorem}[Uniformizing Regular Covering Maps]
\label{uniformize}Let $f:Y\rightarrow X=Y/G$ be a regular covering map
between metrizable spaces. Then:

\begin{enumerate}
\item If $Y$ is given a (compatible) uniform structure such that some root
domain is an entourage in $Y$ (such a uniformity always exists) then

\begin{enumerate}
\item $Y$ has a countable basis $\mathcal{R}$ consisting of root domains.

\item The set of all $f(R)$ with $R\in \mathcal{R}$ is a basis for a uniform
structure on $X$ compatible with the quotient topology.
\end{enumerate}

\item If $X$ has a uniform structure compatible with the quotient topology
such that $f(R)$ is an entourage for some root domain $R$ in $Y$ (such a
uniformity always exists), then:

\begin{enumerate}
\item There is a countable basis $\mathcal{E}$ for the uniformity of $X$
consisting of evenly covered entourages with respect to $R$.

\item For any root domain $R$, the set of all $E_{R}^{\ast }$ for $E\in 
\mathcal{E}$ is an invariant basis for a (compatible) uniform structure on $%
Y $, called the lifted uniformity, such that $G$ acts discretely.

\item The lifted uniformity is unique in the sense that if $R^{\prime }$ is
any other root domain such that $f(R^{\prime })$ is an entourage in $X$ then
the lifted uniformities with respect to $R$ and $R^{\prime }$ are uniformly
equivalent.
\end{enumerate}
\end{enumerate}
\end{theorem}

\begin{proof}
For existence in the first part we may use the \textquotedblleft fine
uniformity\textquotedblright\ on $Y$; since $Y$ is metrizable this simply
consists of all symmetric neighborhoods of the diagonal in $Y\times Y$ (it
is easy to check that this is a compatible uniformity). Root domains are
open symmetric subsets containing the diagonal, hence are entourages. For
Part 1a, just intersect each set in the original countable basis with a
single root domain $R$.

Now suppose that $R\in \mathcal{R}$. Since every $f(R)$ is open in the
quotient topology, to prove the basis statement we need only prove the
\textquotedblleft triangle inequality\textquotedblright . Suppose $S$ is a
root domain with $S^{2}\subset R$. Suppose that $(x,y)\in f(S)^{2}$; so
there exist $(x,z),(y,z)\in f(S)$. This in turn means there are $(x^{\prime
},z^{\prime }),(y^{\prime \prime },z^{\prime \prime })\in S$ such that $%
f(x^{\prime })=x$, $f(z^{\prime })=f(z^{\prime \prime })=z$ and $f(y^{\prime
\prime })=y$. The middle equation means that there is some $g\in G$ such
that $g(z^{\prime \prime })=z^{\prime }$, and since $S$ is invariant,
letting $y^{\prime }:=g(y^{\prime \prime })$ we have $(z^{\prime },y^{\prime
})\in S$. But then $(x^{\prime },y^{\prime })\in S^{2}\subset R$. Therefore $%
(x,y)\in f(R)$, showing that $f(S)^{2}\subset f(R)$.

For Part 2, existence follows from Part 1. For Part 2a we simply intersect
each set in a countable basis for $X$ with $f(R)$ for some root domain $R$.

Now every $E_{R}^{\ast }$ is a symmetric open set containing the diagonal in 
$Y\times Y$, and therefore we need only show the following to prove that the
set of all $E_{R}^{\ast }$ is a basis for a uniform structure: If $F$ is an
evenly covered entourage with $F^{2}\subset E$ then $(F_{R}^{\ast
})^{2}\subset E_{R}^{\ast }$. Let $(a,b)\in (F_{R}^{\ast })^{2}$; so there
are $(a,c),(b,c)\in F_{R}^{\ast }$. This means that $(f(a),f(b)),(f(b),f(c))%
\in f(F_{R}^{\ast })=F$ (see Proposition \ref{domain}.2). In other words, $%
(f(a),f(b))\in F^{2}\subset E$. That is, $(a,b)\in f^{-1}(E)$, and since $%
(a,b)\in F_{R}^{\ast }=f^{-1}(F)\cap R$, $(a,b)\in R$, and by definition $%
(a,b)\in E_{R}^{\ast }$.

Now let $U$ be open in $Y$ and $y\in Y$. Since $f(U)$ is open, there is some
evenly covered entourage $E$ such that $B(f(y),E)\subset f(U)$. By
Proposition \ref{evenly} the restriction $h$ of $f$ to $B(y,E_{R}^{\ast })$
is a bijection onto $B(f(y),E)$. Note that $E_{R}^{\ast }=f^{-1}(E)\cap R$
is open and therefore $V:=B(y,E_{R}^{\ast })\cap U$ is open. Now $%
W:=f(U)\cap f(V)$ is open, contains $f(y)$ and is contained in $B(f(y),E)$.
But then there is an entourage $F\subset E$ in $X$ such that $%
B(f(y),F)\subset W$. Since $h$ is a homeomorphism from $B(y,E_{R}^{\ast })$
onto $B(f(y),E)$, $B(y,F_{R}^{\ast })=h^{-1}(B(f(y),F))$ is contained in $U$.

We already know that the sets $E_{R}^{\ast }$ are invariant, so to finish
the proof of Part 2b we need only show that $G$ acts discretely, which will
follow if we prove that for any evenly covered $E$, $N_{E_{R}^{\ast
}}(G)=\{1\}$. But this is immediate from Proposition \ref{evenly}: If for
some $y\in Y$ and $g\in G$, $(y,g(y))\in E_{R}^{\ast }$ then $g(y)\in
B(y,E_{R}^{\ast })$. But $f(y)=f(g(y))$ and the restriction of $f$ to $%
B(y,E_{R}^{\ast })$ is 1-1, so $y=g(y)$ and $g=1$.

For Part 2c, by taking intersections we may assume that the two root domains
satisfy $R^{\prime }\subset R$. Again by taking intersections, we can take
for our basis elements $E_{R}^{\ast }$ and $E_{R^{\prime }}^{\ast }$ of the
lifted uniformities of $R$ and $R^{\prime }$, respectively, using entourages 
$E$ that are evenly covered by both $R$ and $R^{\prime }$. That is, $%
E\subset f(R^{\prime })$. Suppose that $(a,b)\in E_{R}^{\ast }=f^{-1}(E)\cap
R\subset f^{-1}(f(R^{\prime }))\cap R$, so $(f(a),f(b))\in f(R^{\prime })$
and therefore there exist $(a^{\prime },b^{\prime })\in R^{\prime }$ such
that $f(a^{\prime })=f(a)$ and $f(b^{\prime })=f(b)$. Since $R^{\prime }$ is
invariant, as before we may find $b^{\prime \prime }$ such that $f(b^{\prime
\prime })=f(b)$ and $(a,b^{\prime \prime })\in R^{\prime }$. Since $(a,b)\in
R$, $(b,b^{\prime \prime })\in RR^{\prime }\subset R^{2}\subset P$, where $P$
is the PD-Domain. By Proposition \ref{domain}, $b=b^{\prime \prime }$ and
therefore $(a,b)\in R^{\prime }$. That is, $(a,b)\in E_{R^{\prime }}^{\ast }$%
, showing that $E_{R}^{\ast }\subset E_{R^{\prime }}^{\ast }$. Since $%
R^{\prime }\subset R$, the reverse inclusion is also true, completing the
proof of the theorem.
\end{proof}

\begin{remark}
We summarize how to uniformize a regular covering map $f:Y\rightarrow X=Y/G$%
. First give $Y$ a uniform structure such that some root domain is an
entourage. Take the quotient uniformity on $X=Y/G$, and then lift the
quotient uniformity back to $Y$, from which one obtains a (possibly not
strictly) finer uniformity than the original, which has an invariant basis
and with respect to which $G$ acts discretely. The uniformities on $X$ and $%
Y $ are uniquely determined by the original uniformity on $Y$. When $Y$ is
compact then of course there is only one compatible uniformity on each of $X$
and $Y$, but Theorem \ref{uniformize} still provides the important
information that one can find an invariant basis with respect to which $G$
acts discretely.
\end{remark}

\begin{notation}
In general, root domains, which include all symmetric open sets containing
the diagonal that are contained in a given root domain, need not be
entourages. When $Y$ is a uniform space and $G$ acts discretely and
isomorphically on $Y$, a root domain in $Y$ that is also an entourage will
be called a root entourage.
\end{notation}

The next lemma, which we will use without reference, brings us full circle
concerning uniformizing properly discontinuous actions.

\begin{lemma}
\label{full}Suppose $G$ acts discretely and isomorphically on a uniform
space $Y$. Then

\begin{enumerate}
\item $G$ acts properly discontinuously on the topological space $Y$.

\item $Y$ has a basis of root entourages.

\item The uniform structure on $Y$ is the lift of the quotient uniformity on 
$X=Y/G$ with respect to any root entourage $R$ in $Y$.
\end{enumerate}
\end{lemma}

\begin{proof}
For the first part we need only observe that the orbits via a discrete
action are uniformly discrete, hence closed, so the quotient is Hausdorff.
For the second part, let $E$ be any invariant entourage in $Y$ such that $%
N_{E}(G)=\{1\}$ and suppose that $(x,g(x))\in E$ for some $x\in X$ and $g\in
G$. Then by definition, $g\in N_{E}=\{1\}$, so $E$ is contained in the
PD-domain. But then any invariant entourage $R$ such that $R^{2}\subset E$
is a root domain. The second part follows.

For the third part let $f:Y\rightarrow X=Y/G$ be the quotient map, suppose
that $R$ is a root entourage in $Y$ and $F\subset R$ is an invariant
entourage in $Y$. Since $f(F)^{\ast }=f^{-1}(f(F))\cap R$ is an intersection
of entourages, $f(F)^{\ast }$ is an entourage. We will show $f(F)^{\ast
}\subset F$ (in fact they are equal), completing the proof . If $(x,y)\in
f(F)^{\ast }$ then $(x,y)\in R$ and there is some $(x^{\prime },y^{\prime
})\in F$ such that $f(x^{\prime })=f(x)$ and $f(y^{\prime })=f(y)$. Now
there exists some $g\in G$ such that $g(x^{\prime })=x$, and since $F$ is
invariant, $(x,g(y^{\prime }))\in F\subset R$. Therefore $(y,g(y^{\prime
}))\in R^{2}$, which is contained in the PD-domain. Since $f(y)=f(y^{\prime
})=f(g(y^{\prime }))$, Proposition \ref{domain} implies $y=g(y^{\prime })$.
That is, $(x,y)\in F$.
\end{proof}

\section{Discrete Covers}

\begin{lemma}[General Chain Lifting Lemma]
\label{cl}Suppose $f:Y\rightarrow X=Y/G$ is a discrete cover, $R$ is a root
entourage in $Y$ and $E$ is an evenly covered entourage with respect to $R$.
If $\alpha $ is an $E$-chain in $X$ starting at $x_{0}$ and $\widetilde{x_{0}%
}\in f^{-1}(x_{0})$ then there is a unique $E_{R}^{\ast }$-chain $\widetilde{%
\alpha }$ starting at $\widetilde{x_{0}}$, called the lift of $\alpha $,
such that $f(\widetilde{\alpha })=\alpha $. Moreover, $[\alpha ]_{E}=[\beta
]_{E}$ if and only if $[\widetilde{\alpha }]_{E_{R}^{\ast }}=[\widetilde{%
\beta }]_{E_{R}^{\ast }}$, and in particular if $[\alpha ]_{E}=[\beta ]_{E}$
then $\widetilde{\alpha }$ and $\widetilde{\beta }$ end at the same point in 
$Y$.
\end{lemma}

\begin{proof}
The existence and uniqueness of $\widetilde{\alpha }$ is immediate from
iteration using the fact that the fact that $E$-balls are evenly covered by $%
E_{R}^{\ast }$-balls (Proposition \ref{evenly}). Now suppose that $\alpha
=\{x_{0},...,x_{n}\}$ has unique lift $\widetilde{\alpha }=\{\widetilde{x_{0}%
},...,\widetilde{x_{n}}\}$ and $\beta
=\{x_{0},...,x_{i},x,x_{i+1},...,x_{n}\}$; that is, $\beta $ differs from $%
\alpha $ by the basic move of adding a point $x$. Note that $x,x_{i+1}\in
B(x_{i},E)$. Therefore there are unique points $\widetilde{x},\widetilde{z}%
\in B(\widetilde{x_{i}},E_{R}^{\ast })$ such that $f(\widetilde{z})=x_{i+1}$
and $f(\widetilde{x})=x$. By uniqueness, $\widetilde{z}=\widetilde{x_{i+1}}$%
. On the other hand, the restriction of $f$ to $B(\widetilde{x},E_{R}^{\ast
})$ is a bijection onto $B(x,E)$, which contains $x_{i}$ and $x_{i+1}$.
Therefore there is some $\widetilde{w}\in B(\widetilde{x},E_{R}^{\ast })$
such that $f(\widetilde{w})=x_{i+1}$. But then $\left( \widetilde{x_{i+1}},%
\widetilde{w}\right) \in (E_{R}^{\ast })^{2}=(f^{-1}(E)\cap R)^{2}\subset
R^{2}\subset P$, where $P$ is the PD-domain. Since $f(\widetilde{x_{i+1}}%
)=x_{i+1}=f(\widetilde{w})$, Proposition \ref{domain} implies that $%
\widetilde{x_{i+1}}=\widetilde{w}$. Since $(\widetilde{w},\widetilde{x})\in
E_{R}^{\ast }$, $\left( \widetilde{x_{i+1}},\widetilde{x}\right) \in
E_{R}^{\ast }$. That is, adding $\widetilde{x}$ to $\widetilde{\alpha }$ is
a basic $E_{R}^{\ast }$-move. Removing a point is simply the inverse
operation, completing the proof that if $[\alpha ]_{E}=[\beta ]_{E}$ then $[%
\widetilde{\alpha }]_{E_{R}^{\ast }}=[\widetilde{\beta }]_{E_{R}^{\ast }}$.
The converse follows from the fact that $f(E_{R}^{\ast })=E$.
\end{proof}

\begin{remark}
The above lemma is actually the third version of a Chain Lifting Lemma (see 
\cite{PW} and \cite{PS}); we will discuss this in more detail just prior to
Theorem \ref{quoteq}.
\end{remark}

\begin{proposition}
\label{cpt}If $f:Y\rightarrow X=Y/G$ is a discrete cover between compact
metric spaces then $G$ is finite. Suppose in addition that $X$ is connected.
Then $Y$ has finitely many components $C$, each of which is uniformly open,
and the restriction $f_{C}$ to any of which surjective. In particular $%
f_{C}:C\rightarrow X=C/S_{C}$ is a discrete cover, where $S_{C}$ is the
stabilizer subgroup of $C$.
\end{proposition}

\begin{proof}
If $\{g_{i}\}$ were an infinite collection of distinct elements of $G$ then
taking some $y\in Y$, without loss of generality we could suppose that $%
y_{i}=g_{i}(y)$ is a convergent sequence consisting of distinct points,
since $G$ acts freely. But this is impossible because the action of $G$ is
discrete.

For the next statement we first prove that the restriction of $f$ to $C$ is
surjective. Let $R$ be a root entourage of $Y$ and $E$ be evenly covered
with respect to $R$. Since $X$ is connected, if $w=f(y)$ for some $y\in C$
then for any $x\in X$ there are arbitrarily fine chains in $X$ from $w$ to $%
x $. By the General Chain Lifting Lemma these chains (more precisely those
that are finer than $E$-chains) must lift to arbitrarily fine chains
starting at $y$ ending at some point in $f^{-1}(x)$. But the latter set is
finite, so there must be a point $z\in f^{-1}(x)$ such that there are
arbitrarily fine chains from $y$ to $z$. But this means that $z\in C$, and
since $f(z)=x$ this shows that the restriction of $f$ to any component is
surjective. Therefore each component contains some point in $f^{-1}(x)$ and
since the latter set is finite, there are only finitely many components.

Since there are only finitely many (compact) components, each component $%
C_{i}$ is open. Define $U\subset Y\times Y$ to the be the union of the sets $%
C_{i}\times C_{i}$; then $U$ is an entourage in $Y$ ($Y$ is compact) and as
is easily checked each component is uniformly $U$-open.

The final statement is an immediate consequence of Lemma \ref{cc}.
\end{proof}

\begin{proposition}
\label{liftprop}Let $f:Y\rightarrow X=Y/G$ be a discrete cover and suppose $%
E $ is evenly covered with respect to a root entourage $R$ in $Y$. Then for
any $y\in Y$, $E_{R}^{\ast }$-chain $\alpha $ from $\ast $ to $y$ in $Y$,
and $g\in G$, $g(y)$ is the endpoint of the unique lift of the $E$-chain $%
f(\alpha )$ to $Y$ at $g(\ast )$.
\end{proposition}

\begin{proof}
Since $E_{R}^{\ast }$ is invariant, $g(\alpha )$ is an $E_{R}^{\ast }$-chain
from $g(\ast )$ to $g(y)$. On the other hand, $f(g(\alpha ))=f(\alpha )$ and
therefore $g(\alpha )$ is the unique lift of $f(\alpha )$ to $Y$ at $g(\ast
) $.
\end{proof}

\begin{definition}
\label{pdc}If $f:Y\rightarrow X=Y/G$ is a discrete cover, an entourage $E$
in $X$ is called properly covered if for some root entourage $R$ in $Y$, if $%
E\subset f(R)$ (i.e. $E$ is evenly covered with respect to $R$) and $Y$ is $%
E_{R}^{\ast }$-chain connected. In this case we will say $E$ is properly
covered with respect to $R$. If $X$ has a properly covered entourage $E$
then $f$ is called a proper discrete cover.
\end{definition}

Note that if $Y$ is chain connected then every evenly covered entourage is
properly covered.

\begin{theorem}
\label{eqone}Let $f_{i}:Y_{i}\rightarrow X=Y_{i}/G_{i}$ be discrete covers
of uniform spaces. Suppose $f_{1}\leq f_{2}$, i.e. there is some uniformly
continuous surjection $f:Y_{2}\rightarrow Y_{1}$ such that $f_{2}=f_{1}\circ
f$. Define $\theta _{f}:G_{2}\rightarrow G_{1}$ by letting $\theta _{f}(g)$
be the unique element of $G_{1}$ such that $\theta _{f}(g)(\ast )=f(g(\ast
)) $. Then:

\begin{enumerate}
\item Suppose that $R_{2}$ is a root entourage in $Y_{2}$ and let $%
R_{1}:=f(R_{2})$. Then

\begin{enumerate}
\item $R_{1}$ is a root entourage in $Y_{1}$.

\item If $E$ is evenly covered with respect to $R_{2}$ then $E$ is evenly
covered with respect to $R_{1}$, and $E_{R_{1}}^{\ast }=f(E_{R_{2}}^{\ast })$%
.

\item If $\alpha $ is any $E_{R_{2}}^{\ast }$-chain in $Y_{2}$ from $\ast $
to $y$ and $g\in G_{2}$ then $\theta _{f}(g)(f(y))=f(g(y))$, which is the
endpoint of the unique lift of the $E$-chain $f_{2}(\alpha )$ to $Y_{1}$ at $%
f(g(\ast ))$.

\item If $E$ is properly covered with respect to $R_{2}$ then $E$ is
properly covered with respect to $R_{1}$.
\end{enumerate}

\item $\theta _{f}$ is a surjective homomorphism, and if $E$ is properly
covered with respect to $R_{2}$ we have the compatibility condition $f\circ
g=\theta _{f}(g)\circ f$.

\item $f:Y_{2}\rightarrow Y_{1}=Y_{2}/K_{f}$ is a discrete cover, where $%
K_{f}$ is a normal subgroup of $G_{2}$.

\item $f_{1}$ is equivalent to the induced quotient $\pi
:Y_{2}/K_{f}\rightarrow (Y_{2}/K_{f})/(G_{2}/K_{f})$.
\end{enumerate}
\end{theorem}

\begin{proof}
We start by showing that $f$ is bi-uniformly continuous. Let $R$ be a root
entourage in $Y_{1}$. Since $f$ is uniformly continuous there is a root
entourage $R^{\prime }$ in $Y_{2}$ such that $f(R^{\prime })\subset R$.
Suppose that $F$ is an entourage in $X$ that is evenly covered with respect
to both $R$ and $R^{\prime }$. Since by Theorem \ref{uniformize} the set of
all $F_{R^{\prime }}^{\ast }$ is a basis for the uniform structure on $Y_{2}$%
, the proof will be complete if we show that 
\begin{equation}
f(F_{R^{\prime }}^{\ast })=F_{R}^{\ast }  \label{star2}
\end{equation}%
and hence $f(F_{R^{\prime }}^{\ast })$ is an entourage in $Y_{1}$. Since $%
f(R^{\prime })\subset R$ and $f_{1}\circ f=f_{2}$, 
\begin{equation}
f(F_{R^{\prime }}^{\ast })\subset f(R^{\prime })\cap f(f_{2}^{-1}(F))\subset
R\cap f_{1}^{-1}(F)=F_{R}^{\ast }\text{.}  \label{star}
\end{equation}%
For the reverse inclusion, let $(c,d)\in F_{R}^{\ast }$, i.e. $d\in
B(c,F_{R}^{\ast })$. Since $f$ is surjective there is some $c^{\prime }\in
Y_{2}$ such that $f(c^{\prime })=c$. From Proposition \ref{evenly} we know
that $F$-balls are evenly covered by both $F_{R}^{\ast }$ and $F_{R^{\prime
}}^{\ast }$-balls. If $h_{1}$ denotes the restriction of $f_{1}$ to $%
B(c,F_{R}^{\ast })$, then the restriction of $h_{1}^{-1}\circ f_{2}$ to $%
B(c^{\prime },F_{R^{\prime }})$ is a bijection onto $B(c,F_{R}^{\ast })$.
Therefore there is some $d^{\prime }\in B(c^{\prime },F_{R^{\prime }}^{\ast
})$ such that $h_{1}^{-1}(f_{2}(d^{\prime }))=d$. But by Inclusion (\ref%
{star}), $f(B(c^{\prime },F_{R^{\prime }}^{\ast }))\subset B(c,F_{R}^{\ast
}) $ and therefore $f(d^{\prime })\in B(c,F_{R}^{\ast })$. Since $h_{1}$ is
a bijection and $f_{1}(f(d^{\prime }))=f_{2}(d^{\prime })=f_{1}(d)$, it must
be that $f(d^{\prime })=d$.

Returning to the proof of Part 1a, we now know that $R_{1}$ is an entourage
and therefore we need only show that $R_{1}^{2}$ is contained in the
PD-domain of $Y_{1}$. Suppose that $(y,g_{1}(y))\in R_{1}^{2}$ for some $%
y\in Y_{1}$ and $g_{1}\in G_{1}$. This means that $(y,z),(z,g_{1}(y))\in
R_{1}$ for some $z\in Y_{1}$. Since $R_{1}=f(R_{2})$ this in turn means
there exist $(y^{\prime },z^{\prime }),(z^{\prime \prime },w^{\prime \prime
})\in R_{2}$ such that $f(y^{\prime })=y$, $f(z^{\prime })=f(z^{\prime
\prime })=z$, and $f(w^{\prime \prime })=g_{1}(y)$. Since 
\begin{equation*}
f_{2}(z^{\prime })=f_{1}(f(z^{\prime }))=f_{1}(z)=f_{1}(f(z^{\prime \prime
}))=f_{2}(z^{\prime \prime })\text{,}
\end{equation*}%
there is some $g_{2}\in G_{2}$ such that $g_{2}(z^{\prime \prime
})=z^{\prime }$. Since $R_{2}$ is invariant, letting $w^{\prime
}:=g_{2}(w^{\prime \prime })$, we have that $(w^{\prime },z^{\prime })\in
R_{2}$. Therefore $(y^{\prime },w^{\prime })\in R_{2}^{2}$, which is
contained in the PD-domain $P$ of $Y_{2}$. Now 
\begin{equation*}
f_{2}(y^{\prime })=f_{1}(f(y^{\prime
}))=f_{1}(y)=f_{1}(g_{1}(y))=f_{1}(f(w^{\prime \prime }))=f_{2}(w^{\prime
\prime })=f_{2}(w^{\prime })\text{.}
\end{equation*}%
Therefore there is some $h\in G_{2}$ such that $h(y^{\prime })=w^{\prime }$.
Since $(y^{\prime },w^{\prime })\in P$ this means that $h=1$ and hence $%
y^{\prime }=w^{\prime }$. This in turn means that $y=f(y^{\prime
})=f(w^{\prime })=g_{1}(y)$, so $g_{1}=1$, finishing the proof of Part 1a

For Part 1b, note that if $E\subset f_{2}(R_{2})$ then $E\subset
f_{1}(f(R_{2}))\subset f_{1}(R_{1})$. Next note that Equation (\ref{star2})
depended only on $f(R^{\prime })\subset R$ (and the entourage in question
being evenly covered by both) and therefore it applies to $R^{\prime
}:=R_{2} $ and $R:=R_{1}=f(R_{2})$.

For Part 1c, note that since $f_{2}(\alpha )=f_{2}(g(\alpha ))$ and $%
f_{1}\circ f=f_{2}$, the $E_{R_{1}}^{\ast }$-chain $f(g(\alpha ))$ is the
unique lift of $f_{2}(\alpha )$ to $Y_{1}$ at $f(g(\ast ))$, which ends at $%
f(g(y))$. On the other hand, by Proposition \ref{liftprop}, $\theta
_{f}(g)(f(y))$ is also the endpoint of the unique lift of $f_{1}(f(\alpha
))=f_{2}(\alpha )$ to $Y_{1}$ at $\theta _{f}(g)(\ast )=f(g(\ast ))$.
Therefore, $\theta _{f}(g(f(y)))=f(g(y))$.

If $E$ is properly covered with respect to $R_{2}$ then since $f$\ is
surjective, given $y_{1}\in Y_{1}$ we may find $y_{2}\in f^{-1}(y_{1})$ and
an $E_{R_{2}}^{\ast }$-chain $\alpha $ from $\ast $ to $y_{2}$. But then $%
f(\alpha )$ is an $f(E_{R_{2}}^{\ast })=E_{R_{1}}^{\ast }$-chain from $\ast $
to $y_{1}$. That is, $E$ is properly covered with respect to $R_{1}$.

For the second part, suppose $g_{1},g_{2}\in G_{2}$. By definition of $%
\theta _{f}$ and Part 1c, using $\{\ast \}$ as an $E_{R_{2}}^{\ast }$-chain
from $\ast $ to $\ast $, 
\begin{equation*}
\theta _{f}(g_{1}g_{2})(\ast )=f(g_{1}(g_{2}(\ast )))=\theta
_{f}(g_{1})(f(g_{2}(\ast )))=\theta _{f}(g_{1})(\theta _{f}(g_{2})(\ast ))%
\text{,}
\end{equation*}%
showing that $\theta _{f}$ is a homomorphism (since the action is free).

For surjectivity of $\theta _{f}$, suppose that $h\in G_{1}$. Since $f$ is
surjective there is some $x\in Y_{2}$ such that $f(x)=h(\ast )$. Since $%
f_{2}(x)=f_{1}(f(x))=f_{1}(h(\ast ))=\ast $ there is a unique $h_{2}\in
G_{2} $ such that $h_{2}(\ast )=x$. By definition, $\theta _{f}(h_{1})(\ast
)=f(h_{1}(\ast ))=f(x)=h(\ast )$. Since the action is free, $\theta
_{f}(h_{1})=h$. If $E$ is properly covered with respect to $R_{2}$ then the $%
E_{R_{2}}^{\ast }$-chain $\alpha $ in the statement of Part 1c always
exists, and compatibility follows.

For the last part, since $f$ is bi-uniformly continuous we need only show
that the point preimages of $f$ are the orbits of $K_{f}$ (see Remark \ref%
{quotient} and Proposition \ref{inducedp}), where $K_{f}$ is the kernel of $%
\theta _{f}$. If $x=g(y)$ then by Part 1c, $f(x)=f(g(y))=\theta
_{f}(g)(f(y))=f(y)$. Conversely, suppose that $f(x)=f(y)$. Then 
\begin{equation*}
f_{2}(x)=f_{1}(f(x))=f_{1}(f(y))=f_{2}(y)
\end{equation*}%
and therefore there is some $g\in G_{2}$ such that $x=g(y)$.
\end{proof}

We will now revisit the situation $\phi _{E}:X_{E}\rightarrow X$. By
Proposition 16.1 in \cite{BPUU}, the restriction of $\phi _{E}$ to any $%
E^{\ast }$-ball is 1-1. Now if $([\alpha ]_{E},[\beta ]_{E})\in \left(
E^{\ast }\right) ^{2}$, $[\alpha ]_{E}$ and $[\beta ]_{E}$ lie in some $%
E^{\ast }$-ball and therefore it is impossible for there to be some $g\neq 1$
in $\pi _{E}(X)$ with $g([\alpha ]_{E})=[\beta ]_{E}$. That is, $R=E^{\ast }$
is a root entourage, and in fact it is a canonical choice of root entourage
for $\phi _{E}$. Now if $F\subset E$, an immediate consequence of the
definition of $F^{\ast }$ is that $F^{\ast }=f^{-1}(F)\cap E^{\ast }$. In
other words, in the current terminology, $F^{\ast }=F_{E^{\ast }}^{\ast }$.
For simplicity and consistency we will continue to use the notation $F^{\ast
}$ for the lift of $F$ with respect to the root entourage $E^{\ast }$. In
this case the Chain Lifting Lemma has the following stronger form (\cite{PW}%
): If $\alpha =\{\ast =x_{0},...,x_{n}\}$ is an $F$-chain in $X$ with $%
F\subset E$ then the unique lift of $\alpha $ to $X_{E}$ at $\ast $ consists
of the $F^{\ast }$-chain 
\begin{equation}
\widetilde{\alpha }%
=\{[x_{0}]_{E},[x_{0},x_{1}]_{E},...,[x_{0},...,x_{n}]_{E}=[\alpha ]_{E}\}%
\text{.}  \label{SCL}
\end{equation}%
See also Equation (\ref{clf}), from which this stronger version follows. In
particular, $\widetilde{\alpha }$ ends at $[\alpha ]_{E}$. In order to avoid
confusion among what will ultimately amount to three different chain lifting
lemmas, we will refer to this collection of statements as the \textit{%
Special Chain Lifting Lemma (for }$\phi _{E}$). This statement implies that
every $[\alpha ]_{E}$ is joined to $\ast $ by the $E^{\ast }$-chain $%
\widetilde{\alpha }$, i.e. $X_{E}$ is $E^{\ast }$-connected. We summarize
these observations in a corollary that we will use below without reference:

\begin{corollary}
If $X$ is a uniform space and $E$ is an entourage, then for the map $\phi
_{E}:X_{E}\rightarrow X$, $E^{\ast }$ (as defined in \cite{BPUU}) is a root
entourage and $E$ is properly covered by the entourage $E^{\ast }$, which is
equal to $E_{E^{\ast }}^{\ast }$ as defined in the current paper.
\end{corollary}

\begin{theorem}
\label{quoteq}Let $f:Y\rightarrow X=Y/G$ be a discrete cover with root
entourage $R$ and properly covered entourage $E$. Define $%
f_{E}:X_{E}\rightarrow Y$ as follows. For any $[\alpha ]_{E}\in X_{E}$, let $%
f_{E}([\alpha ]_{E})$ be the endpoint of the unique lift $\widetilde{\alpha }
$ of the $E$-chain $\alpha $ to $Y$ at $\ast $. Then $f_{E}$ is a uniformly
continuous surjection such that $\phi _{E}=f\circ f_{E}$. In particular, $%
f\leq \phi _{E}$ and $f$ is equivalent to the induced quotient $\pi
:X_{E}/K\rightarrow \left( X_{E}/K\right) /\left( \pi _{E}(X)/K\right) $ for
some normal subgroup $K$ of $\pi _{E}(X)$.
\end{theorem}

\begin{proof}
By the Special Chain Lifting Lemma, $f_{E}$ is well defined. Since $E$ is
properly covered, for any $y\in Y$ there is an $E_{R}^{\ast }$-chain $\beta $
from $\ast $ to $y$. Then $f(\beta )$ is an $E$-chain from $\ast $ to $f(y)$
in $X$. But then $\beta $ is the lift of $f(\beta )$ to $Y$ at $\ast $ and
by definition $f([\beta ]_{E})=y$. That is, $f_{E}$ is surjective. By
definition, $\phi _{E}=f\circ f_{E}$.

To prove that $f$ is uniformly continuous we will show that if $F\subset E$
then $f(F^{\ast })\subset F_{R}^{\ast }$ (the latter of which is a basis
element in $Y$ by Lemma \ref{full}). Suppose $([\alpha ]_{E},[\beta
]_{E})\in F^{\ast }$; so there are $E$-chains $\alpha =\{\ast
=x_{0},...,x_{n}=a\}$ and $\beta =\{\ast =x_{0},...,x_{n}=a,b\}$ with $%
(a,b)\in F$. Now $f_{E}([\alpha ]_{E})$ is the endpoint of the unique lift $%
\widetilde{\alpha }=\{\ast =\widetilde{x_{0}},...,\widetilde{a}\}$ of $%
\alpha $ to $Y$ at $\ast $, and by uniqueness the lift $\widetilde{\beta }$
of $\beta $ to $Y$ at $\ast $ must be $\{\ast =\widetilde{x_{0}},...,%
\widetilde{a},\widetilde{b}\}$ for some $\widetilde{b}\in f^{-1}(b)$. So $%
f_{E}([\alpha ]_{E},[\beta ]_{E})=(\widetilde{a},\widetilde{b})\in f^{-1}(F)$%
. Now $\widetilde{\alpha }$ and $\widetilde{\beta }$ are $E_{R}^{\ast }$%
-chains and therefore $(\widetilde{a},\widetilde{b})\in E_{R}^{\ast }\subset
R$. Therefore $f([\alpha ]_{E},[\beta ]_{E})=(\widetilde{a},\widetilde{b}%
)\in f^{-1}(F)\cap R=F_{R}^{\ast }$. This shows that $f\leq \phi _{E}$, and
the last statement is an immediate consequence of Theorem \ref{eqone}.4.
\end{proof}

\begin{remark}
As mentioned previously, $E^{\ast }$ is the canonical choice for a root
entourage in $X_{E}$. From Theorem \ref{quoteq} and Theorem \ref{unique} we
now see that if $f:Y\rightarrow X=Y/G$ is a proper discrete cover then for
any properly covered entourage $E$ there is a canonical choice for a root
entourage in $Y$, namely $R:=f(E^{\ast })$, and $E$ is properly covered with
respect to $R$.
\end{remark}

\begin{theorem}
\label{corr}Suppose $X$ is a continuum and $E$ is an entourage. Then there
exists a finitely generated subgroup $H$ of $\pi _{E}(X)$ that is the
stabilizer subgroup of a uniformly open (hence closed) set $J\subset X_{E}$
containing the basepoint such that the following is true. If $f:Y\rightarrow
X/G$ is a discrete cover of $X$ by a continuum and $E$ is evenly covered
with respect to some root entourage $R$ in $Y$, then for some normal
subgroup $K$ of $H$, $f$ is equivalent to the induced quotient $\pi
:J/K\rightarrow \left( J/K\right) /\left( H/K\right) $.
\end{theorem}

\begin{proof}
To simplify matters for the reader we will outline the proof first. Let $F$
be an entourage in $X$ such that $F^{6}\subset E$ and cover $X$ by a finite
collection of $F$-balls $\mathcal{B}=\{B(c_{i},F)\}_{i=1}^{M}$ that includes 
$c_{0}:=\ast $. Let $L:=\theta _{EF}(\pi _{F}(X))$. That is, $L$ consists of
all $[\lambda ]_{E}\in \pi _{E}(X)$ such that we may take $\lambda $ to be
an $F$-loop. The first step will be to replace any such $\lambda $, up to $%
F^{3}$-homotopy, by an $F^{3}$-loop, each of the points of which is equal to
some $c_{i}$. We will then use the discrete version of a trick from
classical homotopy theory to write, up to $F^{3}$-homotopy, this modified
loop as a product of loops, each of which has at most $2M+1$ points, each
equal to some $c_{i}$. Then compactness and the \textquotedblleft $E^{n}$%
-close $E^{m}$-loops are $E^{m+n}$-homotopic\textquotedblright\ trick
(Remark \ref{close}) show that there are only finitely many such loops, i.e. 
$L$ is contained in some finitely generated group $H$. We then see that $L$
is the stabilizer subgroup of $I:=\phi _{EF}(X_{F})$, and define $J$ to be $%
H(I)=\dbigcup\limits_{k\in H}k(I)$, which has $H$ as its stabilizer
subgroup. It then follows that that the restriction $\rho $ of $\phi _{E}$
to $J$ is a discrete cover of $X$ via the action of the finitely generated
group $H$. Then Theorem \ref{eqone} finishes the proof.

Here are the details. We first claim there is a finite set $\Gamma =\{\gamma
_{i}\}$ in $\pi _{E}(X)$ such that every element of $L$ is a product of
elements of $\Gamma $. In other words, $L$ is contained in the subgroup $H$
of $\pi _{E}(X)$ generated by $\Gamma $.

We subclaim that any $F$-loop $\lambda :=\{\ast =x_{0},...,x_{n}=\ast \}$ is 
$F^{3}$ homotopic to an $F^{3}$-loop 
\begin{equation*}
\lambda ^{\prime }:=\{\ast
=c_{i_{0}},c_{i_{1}},...,c_{i_{n-1}},c_{i_{n}}=\ast \}\text{.}
\end{equation*}%
The iterative argument is similar to one in the proof of Theorem 37 in \cite%
{BPUU}), but for completeness and due to changes in notation, we will prove
it here. For every $0<j<n$, $x_{j}\in B(c_{i_{j}},F)$ for some $c_{i_{j}}$.
Each of the following moves is a basic $F^{2}$-move: 
\begin{equation*}
\lambda \rightarrow \{x_{0},x_{1},x_{1},x_{2},...,x_{n}\}\rightarrow
\{x_{0},x_{1},c_{i_{1}},x_{1},x_{2},...,x_{n}\}
\end{equation*}%
\begin{equation*}
\rightarrow \{x_{0},c_{i_{1}},x_{1},x_{2},...,x_{n}\}\rightarrow
\{x_{0},c_{i_{1}},x_{2},...,x_{n}\}\text{.}
\end{equation*}%
Note also that for any $j$, we have $%
(c_{i_{j}},x_{j}),(x_{j},x_{j+1}),(x_{j+1},c_{i_{j+1}})\in F$ and therefore $%
(c_{i_{j}},x_{j+1})\in F^{2}$ and $(c_{i_{j}},c_{i_{j+1}})\in F^{3}$. Now
each of the following is a basic $F^{3}$-move:%
\begin{equation*}
\{x_{0},c_{i_{1}},x_{2},x_{3}...,x_{n}\}\rightarrow
\{x_{0},c_{i_{1}},c_{i_{2}},x_{2},x_{3}...,x_{n}\}
\end{equation*}%
\begin{equation*}
\rightarrow \{x_{0},c_{i_{1}},x_{2},c_{i_{2}},x_{3},...,x_{n}\}\text{.}
\end{equation*}%
Proceeding iteratively finishes the proof of the subclaim.

Now suppose that $n>M$, i.e. there is a repeated point in $\lambda ^{\prime
} $. That is, for some smallest $i_{j}$ and $i_{k}>i_{j}$, $%
c_{i_{j}}=c_{i_{k}} $. Then $\lambda ^{\prime }$ is $F^{3}$-homotopic to the
concatenation of 
\begin{equation*}
\lambda ^{\prime \prime }=\{\ast
=c_{i_{0}},...,c_{i_{j}},c_{i_{j}+1},...,c_{i_{k}}=c_{i_{j}},c_{i_{j}-1},...,c_{i_{0}}=\ast \}
\end{equation*}%
and 
\begin{equation*}
\lambda ^{\prime \prime \prime }=\{\ast
=c_{i_{0}},...,c_{i_{j}}=c_{i_{k}},c_{i_{k+1}},...,c_{i_{n}}=\ast \}\text{.}
\end{equation*}%
Here the discrete homotopies are all \textquotedblleft
retractions\textquotedblright\ of the form $[\tau \ast \overline{\tau }%
]_{F^{3}}=[t]_{F^{3}}$, where $t$ is the endpoint of $\tau $. That is, $%
[\lambda ^{\prime }]_{F^{3}}=[\lambda ^{\prime \prime }]_{F^{3}}[\lambda
^{\prime \prime \prime }]_{F^{3}}$, where $\lambda ^{\prime \prime }$ has at
most $2M+1$ points and $\lambda ^{\prime \prime \prime }$ has strictly fewer
points than $\lambda ^{\prime \prime }$. Repeating this process shows that $%
[\lambda ]_{F^{3}}$ is a product of $F^{3}$-equivalence classes of $F^{3}$%
-loops having at most $2M+1$ points. The proof of our claim will be finished
if we show that the set $\Gamma $ of $E$-homotopy equivalence classes of
such loops is finite. Suppose $\{\lambda _{m}\}_{m=1}^{\infty }$ is a
sequence of $F^{3}$-loops each having at most $2M+1$ points such that no two
are $F^{3}$-homotopic. By duplicating points if necessary (which doesn't
change the $F^{3}$-homotopy class) we can suppose that $\lambda _{m}=\{\ast
=z_{0}^{m},...,z_{2M}^{m}=\ast \}$. Taking a subsequence if necessary we may
assume that for all $j$, $z_{j}^{m}\rightarrow z_{j}$ for some $z_{j}$. Then 
$\lambda =\{\ast =z_{0},...,z_{2K}=\ast \}$ is an $\overline{F^{3}}$-loop
and hence is an $F^{5}$-loop. In fact, since $(z_{i},z_{i+1})$ is in the
closure of $F^{3}$, we may find $(z_{i}^{\prime },z_{i+1}^{\prime })\in
F^{3} $ such that $z_{i}^{\prime }\in B(z_{i},F)$ and $z_{i+1}^{\prime }\in
B(z_{i+1},F)$. But then $(z_{i},z_{i+1})\in F^{5}$.

Now for all $j$ and all large enough $m$, $(z_{j}^{m},z_{j})\in F$.
Referring again to Remark \ref{close} we see that that $\lambda _{m}$ is $%
F^{6}$-homotopic to $\lambda $ and therefore for all large $m,n$, $\lambda
_{m}$ is $F^{6}$-homotopic, hence $E$-homotopic, to $\lambda _{n}$. This is
a contradiction, proving the claim that $\Gamma $ is finite.

Now define $I:=\phi _{EF}(X_{F})$. We claim that $I$ is the $F^{\ast }$%
-chain component of $\ast $ in $X_{E}$. Suppose that $[\alpha ]_{E}\in I$,
which means we can assume that $\alpha $ is an $F$-chain, and suppose that
for some $[\beta ]_{E}$, $([\alpha ]_{E},[\beta ]_{E})\in F^{\ast }$. This
means $[\overline{\alpha }\ast \beta ]_{E}=[a,b]_{E}$ with $(a,b)\in F$.
Note that $[\beta ]_{E}=[\alpha \ast \{a,b\}]_{E}$ and $\alpha \ast \{a,b\}$
is an $F$-chain, showing that $[\beta ]_{E}\in I$. That is, $I$ is uniformly 
$F^{\ast }$-open and therefore contains the $F^{\ast }$-chain component of $%
\ast $ in $X_{E}$. On the other hand, if $[\alpha ]_{E}\in I$, then by the
Special Chain Lifting Lemma, the lift of $\alpha $ to $X_{E}$ at the
basepoint is an $F^{\ast }$-chain that ends at $[\alpha ]_{E}$. That is, $I$
is $F^{\ast }$-chain connected and is therefore the $F^{\ast }$-chain
component of $\ast $ in $X_{E}$.

We next claim that $L$ is the stabilizer subgroup of $I$. In fact, if $%
[\lambda ]_{E}\in L$, this means that $\lambda $ is an $F$-chain and $%
[\lambda ]_{E}([\ast ]_{E})=[\lambda ]_{E}\in I$. The claim now follows from
Lemma \ref{cc}.

Now let $J:=H(I)=\{g(x):g\in H$ and $x\in I\}$. Put another way, $J$ is the
union of the translates $g(I)$ with $g\in H$. Note that since $F^{\ast }$ is
invariant with respect to $\pi _{E}(X)$, and $I$ is uniformly $F^{\ast }$%
-open, $J$ is uniformly $F^{\ast }$-open. Since $X$ is (chain) connected, $%
\phi _{F}:X_{F}\rightarrow X$ is onto, and since $X=\phi _{F}(X_{F})=\phi
_{E}(\phi _{F}(X_{F}))=\phi _{E}(I)\subset \phi _{E}(J)$, the restriction $%
\rho $ of $\phi _{E}$ to $J$ is onto.

We next claim that $H$ is the stabilizer subgroup $S_{J}$ of $J$ and $J$ has
Property ST, so we may apply Lemma \ref{cc} to conclude that $\rho
:J\rightarrow X=J/H$ is a discrete cover. If $x=g(y)\in J$ with $y\in I$ and 
$g\in H$, and $h\in H$, then $h(x)=hg(y)\in J$, so $h(J)\subset J$. Now let $%
w\in J$. Since $h$ is a bijection there is some $z\in X_{E}$ such that $%
h(z)=w$, and hence $h^{-1}(w)=z$. But by what we just showed, $z\in J$,
showing that $J=h(J)$, i.e. $H\subset S_{J}$.

We next note that $R:=E^{\ast }\cap \left( J\times J\right) $ is a root
domain for $\rho $ (since as we have previously observed, $E^{\ast }$ is a
root domain for $\phi _{E}$). We claim that $F$ is evenly covered with
respect to $R$. In fact, suppose that $(a,b)\in F$. Since $\phi _{E}(F^{\ast
})=F$ there exist $([\alpha ]_{E},[\beta ]_{E})\in F^{\ast }$ such that $%
(a,b)=(\phi _{E}([\alpha ]_{E},[\beta ]_{E})$. Since $\rho $ is surjective
there is some $z\in J$ such that $\rho (z)=a$ and therefore there is some $%
g\in \pi _{E}(X)$ such that $g([\alpha ]_{E})=z$. Since $F^{\ast }$ is
invariant, $(z,g([\beta ]_{E})\in F^{\ast }$. Since $J$ is uniformly $%
F^{\ast }$-open, $g([\beta ]_{E})\in J$. Therefore $(z,g([\beta ]_{E}))\in
F^{\ast }\cap (J\times J)\subset R$, showing that $F\subset \rho (R)$.

Now suppose $k\in \pi _{E}(X)$ and $k(a)=b$ for some $a,b\in J$ (which
includes the case when $k\in S_{J}$). Suppose $a=h_{1}(a^{\prime })$ and $%
b=h_{2}(b^{\prime })$ with $h_{1},h_{2}\in H$ and $a^{\prime },b^{\prime
}\in I$. That is, $k(h_{1}(a^{\prime }))=h_{2}b^{\prime }\Rightarrow
h_{2}^{-1}\circ k\circ h_{1}(a^{\prime })=b^{\prime }$; so $%
h_{2}^{-1}kh_{1}=\lambda \in L\subset H$. That is, $k=h_{2}\lambda
h_{1}^{-1}\in H$, proving both that $H=S_{J}$ and that $J$ has Property ST.

Note that since $F\subset E$, $F$ is also evenly covered with respect to $f$%
. Let $f_{J}:J\rightarrow Y$ be the restriction of the discrete covering map 
$f_{E}:X_{E}\rightarrow Y$ from Theorem \ref{quoteq}. That is, $%
f_{E}([\alpha ]_{E})$ is the endpoint of the lift of $\alpha $ to $Y$ at $%
\ast $. We claim that $f_{J}$ is surjective. In fact, there is some $%
F_{R}^{\ast }$-chain $\beta $ from $\ast $ to any $y\in Y$. But then $\alpha
:=f(\beta )$ is an $F$-chain, and $\alpha $ has a unique lift to an $F^{\ast
}$-chain $\widetilde{\alpha }$ to $X_{E}$ at $\ast $, which ends at $[\alpha
]_{E}$. But since $I$ is uniformly $F^{\ast }$-open, $\widetilde{\alpha }$
must stay in $I$ and therefore $[\alpha ]_{E}\in I\subset J$. By definition
of $f_{E}$, $f_{E}([\alpha ]_{E})=y$ and so $f_{J}$ is surjective. Since $%
f_{J}$ is uniformly continuous and satisfies $f\circ f_{J}=\rho $, we have
that $f\leq \rho $. The proof is now finished by Theorem \ref{eqone}.
\end{proof}

\begin{proof}[Proof of Theorem \protect\ref{factor}]
Suppose that a continuum $X$ has $m>n!$ $n$-fold covers by continuua. Let $E$
be an entourage in $X$ that is evenly covered by all of them. Then the
finitely generated group $H$ from Theorem \ref{corr} has $m>n!$ normal
subgroups of index $n$, contradicting a classical algebraic theorem proved
by Hall (\cite{H}, Section 2).
\end{proof}

\begin{remark}
\label{3mfd}There is a natural question: when is a covering map of the form $%
\phi _{E}:X_{E}\rightarrow X$, i.e. when can we take $K_{f_{E}}$ to be the
trivial group? This question is mostly answered for compact smooth manifolds 
\cite{PS}: For any compact smooth manifold $M$ of dimension at least 3, the
\textquotedblleft entourage covers\textquotedblright , i.e. $\phi
_{E}:M_{E}\rightarrow M$, where $E$ is an entourage $E$ has a natural
property that ensures $M_{E}$ is connected, are precisely those covering
maps corresponding to subgroups $G$ of $\pi _{1}(M)$ that are the normal
closures of finite sets (Theorem 8, \cite{PS}). For the one and only compact
manifold of dimension 1, the only entourage covers are the trivial cover and
the universal cover (Example 68, \cite{PS}). For compact surfaces, the
characterization of entourage covers is an open question.
\end{remark}

\section{Generalized Regular Covering Maps}

\begin{definition}
\label{guc}If $G$ is complete and acts isomorphically and prodiscretely
(resp. discretely) on a uniform space $X$ then the quotient map $\phi
:X\rightarrow X/G$ is called a \textit{generalized regular covering map}
(resp. \textit{a discrete regular covering map or simply a discrete cover})
with \textit{deck group} $G$.
\end{definition}

Note that if $G$ acts discretely then $G$, as a (uniformly) discrete group,
is automatically complete. For those reading \cite{PQ} we note that what we
more appropriately are calling \textquotedblleft generalized regular
covering maps\textquotedblright\ in the present paper were simply called
\textquotedblleft covers of uniform spaces\textquotedblright\ in \cite{PQ}.
We also note that Lemma 39 in that paper inadvertently leaves out the word
\textquotedblleft closed\textquotedblright\ prior to \textquotedblleft
subgroup\textquotedblright .

We will need the following theorem summarizing some results about inverse
systems of quotients, much of it derived from \cite{PQ}. We note that in the
special case when $G$ is the fundamental group of a Poincar\'{e} space $Y$, $%
X=\widetilde{Y}$, and the indexing set is $\mathbb{N}$, McCord (\cite{Mc})
proved Theorem \ref{snark}.2 (Theorem 5.8) and the first part of Theorem \ref%
{snark}.4 (Theorem 5.12). Those proofs in fact do not need the stronger
assumptions of \cite{Mc} and carry over directly to the current setting.

\begin{theorem}
\label{snark}Suppose $G$ acts isomorphically on a uniform space $X$. Suppose
that $\mathcal{K}=\{K_{\alpha }\}$ is a directed set of normal subgroups of $%
G$ ordered by reverse inclusion. Let $X_{\alpha }:=X/K_{\alpha }$ and define 
$\phi _{\alpha \beta }:X_{\beta }\rightarrow X_{\alpha }$ by $\phi _{\alpha
\beta }(K_{\beta }x)=K_{\alpha }x$. Define $G_{\alpha }:=G/K_{\alpha }$ and $%
\theta _{\alpha \beta }:G_{\beta }\rightarrow G_{\alpha }$ by $\theta
_{\alpha \beta }(gK_{\beta })=gK_{\alpha }$. Finally, let $\phi
:X\rightarrow \lim \underleftarrow{X_{\alpha }}:=\overline{X}$ and $\theta
:G\rightarrow \underleftarrow{G_{\alpha }}:=\overline{G}$ be the unique maps
determined by the quotient mappings $\phi _{\alpha }:X\rightarrow X_{\alpha
}=X/G_{\alpha }$ and $\theta _{\alpha }:G\rightarrow G_{\alpha }$. Then

\begin{enumerate}
\item The systems $\{X_{\alpha },\phi _{\alpha \beta }\}$ and $\{G_{\alpha
},\theta _{\alpha \beta }\}$ comprise an isomorphic inverse system of
quotients. (As defined in \cite{PQ}, this means that the natural
compatibility condition $\phi _{\alpha \beta }\circ g=\theta _{\alpha \beta
}(g)\circ \phi _{\alpha \beta }$ is satisfied for all $\alpha \leq \beta $
and $g\in G_{\beta }$.) Moreover, the projections $\phi ^{\alpha }:\overline{%
X}\rightarrow X_{\alpha }$ and $\theta ^{\alpha }:\overline{G}\rightarrow G$
are surjective for all $\alpha $.

\item The map $\phi $ is uniformly continuous with dense image in $\overline{%
X}$, with the inverse limit uniformity.

\item The map $\theta $ is a homomorphism with dense image in $\overline{G}$%
, with the inverse limit topology, which is the same as the topology of
uniform convergence.

\item If $G$ acts freely then the following are equivalent:

\begin{enumerate}
\item $\cap _{\alpha }K_{\alpha }=\{1\}$

\item The maps $\phi $ and $\theta $ are injective.

\item The map $\phi $ or the map $\theta $ is injective.
\end{enumerate}

\item If every collection of orbits $\{K_{\alpha }x_{\alpha }\}$ such that $%
K_{\alpha }x_{\alpha }\subset K_{\beta }x_{\beta }$ whenever $\beta \leq
\alpha $ has non-empty intersection then $\phi $ and $\theta $ are
surjective. This is in particular always true if the orbits are compact.

\item If $\phi $ is surjective and for every entourage $F$ in $X$ there is
some $K_{\alpha }$ such that $K_{\alpha }\subset U_{F}(G)$, and $G$ acts
freely, then $\phi $ is equivalent to the quotient map $\pi :X\rightarrow
X/K $, where $K=\cap _{\alpha }K_{\alpha }$. If in addition $G$ acts
discretely (resp. each $K_{\alpha }$ is complete and $G$ acts prodiscretely)
then $\phi $ is a discrete (resp. generalized) regular covering map.
\end{enumerate}
\end{theorem}

\begin{proof}
That the systems comprise an isomorphic inverse system of quotients follows
from Proposition 47 in \cite{PQ}. Since for every $\beta \leq \alpha $, $%
\phi _{\alpha \beta }\circ \phi ^{\beta }=\phi _{\alpha }$ and $\theta
_{\alpha \beta }\circ \theta ^{\beta }=\theta _{\alpha }$, the mappings $%
\phi $ and $\theta $ are guaranteed by the universal property of the inverse
limit and defined by $\phi (x):=(K_{\alpha }x)$ and $\theta (g):=(gK_{\alpha
})$. We have for all $\alpha $, $\phi ^{\alpha }\circ \phi =\phi _{\alpha }$
and $\theta ^{\alpha }\circ \theta =\theta _{\alpha }$ and since $\phi
_{\alpha }$ and $\theta _{\alpha }$ are surjective, so are $\phi ^{\alpha }$
and $\theta ^{\alpha }$. That $\phi $ is uniformly continuous was proved in
Proposition 47.3 in \cite{PQ}. Part 2 is a well-known (in modern times)
consequence of the surjectivity of the maps $\phi ^{\alpha }$; and the proof
used by McCord (Theorem 5.8, \cite{Mc}) is now standard.

Likewise, that $\theta $ is a homomorphism with dense image in the inverse
limit topology is standard. The proof that the two topologies on $G$ are the
same may be found in the proof of Theorem 44, \cite{PQ} (see the proof of
this statement for the group $K_{\beta }$ in the penultimate paragraph).

For the fourth part, McCord's proof works here as well (this has only to do
with inverse limits of sets and abstract groups acting on them by
bijections), see also Proposition 45 in \cite{PQ}. The fifth statement
follows from Proposition 45.3 in \cite{PQ} (plus the well-known fact due to
Cantor that the intersection of a nested collection of non-empty compact
sets is non-empty).

The sixth part significantly improves Proposition 47.4 in \cite{PQ}; we
adapt the proof here. We start by showing that $\phi $ is bi-uniformly
continuous. Let $E$ be an entourage in $X$; we will show that an entourage
of the form $\left( \phi ^{\alpha }\right) ^{-1}(\phi _{\alpha }(F))$ is
contained in $\phi (E)$. Let $F$ be an entourage such that $F^{3}\subset E$
and let $\alpha $ be such that $K_{\alpha }\subset U_{F}(G)$. The elements
of $\left( \phi ^{\alpha }\right) ^{-1}(\phi _{\alpha }(F))$ are of the form 
$((K_{\beta }x,K_{\beta }y))$ with $(K_{\alpha }x,K_{\alpha }y)\in \phi
_{\alpha }(F)$. This means that for some $h,k\in K_{\alpha }$, $%
(h(x),k(y))\in F$. Since $h,k\in U_{F}(G)$ we also have $(h(x),x),(k(y),y)%
\in F$. Therefore $(x,y)\in F^{3}\subset E$. But then $((K_{\beta
}x,K_{\beta }y))=\phi ((x,y))\in \phi (E)$, completing the proof that $\phi $
is bi-uniformly continuous.

Next note that by definition, $\phi (x)=\phi (y)$ if and only if for all $%
\alpha $, $\phi _{\alpha }(x)=\phi _{\alpha }(y)$. This in turn is
equivalent to the fact that for some $g_{\alpha }\in K_{\alpha }$, $%
g_{\alpha }(x)=y$. Since $G$ acts freely this implies that $g_{\alpha
}=g_{\beta }$ for all $\alpha ,\beta $ and $g_{\alpha }\in K$ with $%
g_{\alpha }(x)=y$. Conversely, if $g\in K$ with $g(x)=y$ then for every $%
\alpha $, $g\in K_{\alpha }$ and $g(x)=y$. That is, the orbits of $K$ are
precisely the point pre-images of $\phi $ and since $\phi $ is bi-uniformly
continuous, the first statement in the sixth part is proved (see Remark \ref%
{quotient}). For the very last statement, the only question concerns the
completeness of $K$, which is automatic for discrete actions.
\end{proof}

\begin{definition}
\label{reso}Let $X,G,\mathcal{K}$ be as in the statement of Theorem \ref%
{snark}, and suppose that $G\in \mathcal{K}$ and all groups in $\mathcal{K}$
are complete. Then the resulting inverse system is called the $\mathcal{K}$%
-resolution of the quotient $\pi :X\rightarrow X/G$ and $\overline{X}$ is
called the $\mathcal{K}$-completion of $X$. We have the following special
cases:

\begin{enumerate}
\item Let $\mathcal{K}$ be the collection of all closed normal subgroups of $%
G$ of finite index. We refer to $\overline{X}$ as the profinite completion
of $X$ (with respect to $G$) and the $\mathcal{K}$-resolution as the
profinite resolution of $\pi $.

\item When $\pi :X\rightarrow X/G$ is a generalized regular covering map and 
$\mathcal{K}$ is the collection of $N_{E}$ for entourages $E$ in some
basis,in $X$ (which are open hence closed and therefore complete) then $\phi 
$ is uniform homeomorphism. We may identify the quotient map $\pi
:X\rightarrow X/G$ with the quotient map $\overline{\pi }:\overline{X}%
\rightarrow \overline{G}$ and we simply call the $\mathcal{K}$-resolution
the resolution of $\pi $, and each of the induced quotients is a discrete
cover (see Theorem 48 in \cite{PQ}).

\item When $\mathcal{K}$ is the collection of all closed normal subgroups of 
$G$ then we refer to the $\mathcal{K}$-resolution as the full resolution
(although we do not need this concept for this paper).
\end{enumerate}
\end{definition}

\begin{example}
Let $\mathbb{Z}$ act on $\mathbb{R}$ in the usual way with quotient the
circle $S^{1}$. Then the profinite resolution of the quotient $\pi :\mathbb{%
R\rightarrow R}/\mathbb{Z}=S^{1}$ is the quotient $\overline{\pi }:\Sigma
\rightarrow S^{1}\ $via the action of the profinite completion of $\mathbb{Z}
$ on the so-called universal solenoid $\Sigma $, which is the inverse limit
of all compact regular covers of $S^{1}$. This quotient is precisely the
\textquotedblleft K-universal cover\textquotedblright\ of $S^{1}$ considered
as a compact topological group in \cite{BPCG} and is also the compact
universal cover in the sense of the present paper. See also Example \ref%
{solenoid}.
\end{example}

\begin{proposition}
\label{uni!}Suppose that $f:Y\rightarrow X=Y/G$ is a generalized regular
covering map and $Z$ is a chain connected uniform space. If there are
(possibly not basepoint-preserving) uniformly continuous functions $%
g,g^{\prime }:Z\rightarrow Y$ such that $h:=f\circ g=f\circ g^{\prime }$
then $g=g^{\prime }$ (mod $G$). Moreover, $g^{\prime }=g$ if and only if for
any choice $\ast $ of basepoint in $Z$ such that $h$ is basepoint
preserving, $g(\ast )=g^{\prime }(\ast )$.
\end{proposition}

\begin{proof}
We begin with the assumption that we have chosen a basepoint $\ast $ in $Z$
such that $g$ and $g^{\prime }$ are both basepoint preserving, and will show
that $g=g^{\prime }$. Suppose first that $f$ is a discrete covering map and
let $E$ be an evenly covered entourage in $X$ with respect to a root
entourage $R$ in $Y$. Let $F$ be an entourage in $Z$ such that $%
g(F),g^{\prime }(F)\subset E_{R}^{\ast }$. For any $z\in Z$, let $\alpha $
be be an $F$-chain from $\ast $ to $z$. Since $f(E_{R}^{\ast })=E$, $%
h(\alpha )$ is an $E$-chain in $X$, which therefore has a unique lift $%
\widetilde{\alpha }$ to $Y$ at $\ast $. Since $g(\alpha )$ and $g^{\prime
}(\alpha )$ are $E_{R}^{\ast }$-chains and hence lifts of $h(\alpha )$,
they, and their endpoints $g(z)$ and $g^{\prime }(z)$, must be equal.

Now suppose that $f$ is a generalized regular covering map; so the
resolution of $f$ is an inverse system spaces $\{Y_{i},f_{ij}\}$ such that
the induced quotients $f_{i}:Y_{i}=Y/K_{i}\rightarrow X=Y_{i}/G_{i}$ are
discrete covering maps. Let the basepoint $\ast $ in $Y_{i}$ be $\pi
_{i}(\ast )$, and note that the projections $\pi _{i}:Y\rightarrow X_{i}$
are surjective and uniformly continuous. We have basepoint-preserving
surjective compositions $g_{i}:=\pi _{i}\circ g:Z\rightarrow Y_{i}$ and $%
g_{i}^{\prime }:=\pi _{i}\circ g\prime :Z\rightarrow Y_{i}$, with $%
f_{i}\circ g_{i}=f=f_{i}\circ g_{i}^{\prime }$. From what we proved above, $%
g_{i}=g_{i}^{\prime }$ for all $i$. Denoting elements of $Y$ by $(y_{i})$
with $y_{i}\in Y_{i}$, we have that for all $i$ and $z\in Z$, $%
g(z)=(g_{i}(z))=(g_{i}^{\prime }(z))=g^{\prime }(z)$.
\end{proof}

Next, suppose that $g$ is basepoint preserving but $g^{\prime }(\ast )=\ast
^{\prime }\in f^{-1}(\ast )$. There is a unique $k\in G$ such that $k(\ast
)=\ast ^{\prime }$. Now $k\circ g^{\prime }(\ast )=\ast $ and by what we
proved above, $k\circ g^{\prime }=g$; that is, $g=g^{\prime }$ (mod $G$).
The last part of the proposition is now immediate.

\begin{corollary}
\label{eqeq}Let $f_{i}:Y_{i}\rightarrow X=Y/G_{i}$ be generalized regular
covering maps. Then the actions of $G_{1}$ and $G_{2}$ are equivalent if and
only if that there is a uniform homeomorphism $h:Y_{1}\rightarrow Y_{2}$
such that $f_{2}\circ h=f_{1}$ (i.e. $f_{1}$ and $f_{2}$ are equivalent in
the classical sense for regular covering maps).
\end{corollary}

\begin{remark}
Theorem \ref{compo} below answers an open question from 2007 (\cite{BPUU},
p. 1751).
\end{remark}

\begin{theorem}
\label{compo}The composition of discrete covers (resp. generalized regular
covering maps) between chain connected metrizable uniform spaces is a
discrete cover (resp. generalized regular covering map). More precisely,
suppose that $g:Z\rightarrow Y=Z/H$ and $f:Y\rightarrow X=Y/G$ are discrete
covers (resp. generalized regular covering maps) and let $h:=f\circ g$. Then
there exists a complete group $K$ of uniform homeomorphisms of $Z$ that
contains $H$ as a normal subgroup such that $h:Z\rightarrow X=Z/K$ is a
discrete cover (resp. generalized regular covering map).
\end{theorem}

\begin{proof}
Choose basepoints so that both maps are basepoint-preserving. To begin with
we assume that $f:Y\rightarrow X=Y/G$ is any quotient map via a free
isomorphic action and $g:Z\rightarrow Y=Z/H$ is a discrete cover. Let $E$ be
an invariant entourage in $Y$ that is evenly covered with respect to some
root entourage $R$ in $Z$. We note for once and for all that by Lemma \ref%
{full} and Theorem \ref{uniformize} the set of all such $E$ is a basis for
the uniform structure on $Y$ and the set of all $E_{R}^{\ast }$ is an
invariant basis for the uniform structure on $Z$. Suppose that $x\in
h^{-1}(\ast )$. Define a function $k_{x}:Z\rightarrow Z$ as follows. Since $%
x\in h^{-1}(\ast )$, $g(x)\in f^{-1}(\ast )$ and therefore there is a unique 
$j_{x}\in G$ such that $j_{x}(\ast )=g(x)$. For any $z\in Z$, let $\beta $
be an $E_{R}^{\ast }$-chain from $\ast $ to $z$. Define $k_{x}(z)$ to be the
endpoint of the unique lift $\widetilde{\beta }$ of $j_{x}(g(\beta ))$ to $Z$
at $x$. We will show that the set $K$ of all such $k_{x}$ is a group of
well-defined uniform homeomorphisms acting freely and isomorphically on $Z$,
which contains $H$ as a normal subgroup, and such that $h$ is the quotient
map $h:Z\rightarrow X=Z/K$. This will involve a series of claims, beginning
with the claim that $k_{x}$ is well-defined.

Suppose that $\beta ^{\prime }$ is another $E_{R}^{\ast }$-chain from $\ast $
to $z$. Then $g(\beta )$ and $g(\beta ^{\prime })$ end at the same point,
and therefore $k_{x}(g(\beta ))$ and $k_{x}(g(\beta ^{\prime }))$ end at the
same point $y$. Let $z^{\prime },z^{\prime \prime }$ be the endpoints of the
lifts $\widetilde{\beta }$ and $\widetilde{\beta ^{\prime }}$ of $g(\beta )$
and $g(\beta ^{\prime })$, respectively, to $X$ at $x$. Since $g(z^{\prime
})=y=g(z^{\prime \prime })$, there is some $m\in H$ such that $m(z^{\prime
})=z^{\prime \prime }$. By Proposition \ref{liftprop}, if $\tau $ is any $%
E_{R}^{\ast }$-chain from $\ast $ to $x$, then $m(\ast )$ is the endpoint of
the lift $\kappa $ of $g\left( \overline{\widetilde{\beta }}\ast \overline{%
\tau }\right) $ at $z^{\prime \prime }$. By uniqueness, the lift of $g\left( 
\overline{\widetilde{\beta }}\right) $ at $z^{\prime \prime }$ must be $%
\overline{\widetilde{\beta ^{\prime }}}$, which ends at $x$. By uniqueness
again, we have that $\kappa =\overline{\widetilde{\beta ^{\prime }}}\ast 
\overline{\tau }$, which ends at $\ast $. That is, $m(\ast )=\ast $, which
implies $m=1$ and $z^{\prime }=z^{\prime \prime }$. It is immediate from the
definition of $k_{x}$ that $h=h\circ k_{x}$, and we also have the following
compatibility condition for any $x\in h^{-1}(\ast )$ and $z\in Z$: 
\begin{equation}
g(k_{x}(z))=j_{x}(g(z))\text{.}  \label{compat}
\end{equation}%
In fact, both sides of this equation are readily seen to be the endpoint of $%
j_{x}(g(\beta ))$.

Letting $x^{\prime }$ be the endpoint of the lift of $g(\overline{\tau })$
to $Z$ at $\ast $, we see by uniqueness of lifts that $k_{x^{\prime }}$ is
an inverse function to $k_{x}$, proving that $k_{x}$ is a bijection. We next
claim that for any $x_{1}\in h^{-1}(\ast )$, $k_{x_{1}}\circ
k_{x}=k_{k_{x_{1}}(x)}$, showing that $K$ is a group. With $z$ and $\beta $
as above, $\tau \ast \widetilde{\beta }$ is an $E_{R}^{\ast }$-chain from $%
\ast $ to $k_{x}(z)$. Therefore $k_{x_{1}}(k_{x}(z))$ is the endpoint of the
lift of 
\begin{equation*}
j_{x_{1}}(g(\tau \ast \widetilde{\beta }))=j_{x_{1}}(g(\tau ))\ast
j_{x_{1}}(j_{x}(g(\beta )))
\end{equation*}%
to $Z$ at $x_{1}$. This in turn is the endpoint of the unique lift of $%
j_{x_{1}}(j_{x}(g(\beta )))=j_{j_{x_{1}}(x)}(g(\beta ))$ ($G$ acts freely)
to $Z$ at $k_{x_{1}}(x)$. On the other hand, $h(k_{x_{1}}(x))=h(x)=\ast $,
and therefore $k_{k_{x_{1}}(x)}$ is defined. And by definition, $%
k_{k_{x_{1}}(x)}(z)$ is also the endpoint of the lift of $%
j_{k_{x_{1}}(x)}(g(\beta ))$ to $Z$ at $k_{x_{1}}(x)$.

Note that by Proposition \ref{liftprop}, $H$ is characterized as those $m\in
K$ such $g(m(\ast ))=\ast $. By Proposition \ref{evenly} the restriction $%
g_{z}$ of $g$ to any $B(z,E_{R}^{\ast })$ is a bijection onto $B(g(z),E)$.
We claim that the restriction $k_{x}^{z}$ of $k_{x}$ to any $B(z,E_{R}^{\ast
})$ satisfies the equation 
\begin{equation}
k_{x}^{z}=g_{k_{x}(z)}^{-1}\circ j_{x}\circ g_{z}\text{.}  \label{essential}
\end{equation}%
Since $E$ is invariant with respect to $j_{x}$ this will mean that the
restriction of $k_{x}$ to any $E_{R}^{\ast }$-ball is a bijection onto an $%
E_{R}^{\ast }$-ball, showing that $E_{R}^{\ast }$ is invariant with respect
to $k_{x}$. In fact, if $w\in B(z,E_{R}^{\ast })$ then since $k_{x}$ is
well-defined we may use any $E_{R}^{\ast }$-chain $\beta $ from $\ast $ to $%
z $, concatenated with $(z,w)$, as our $E_{R}^{\ast }$-chain from $\ast $ to 
$w $. Now the fact that $k_{x}(w)=g_{k_{x}(z)}^{-1}\circ j_{x}\circ g_{z}(w)$
is immediate from the definition of $k_{x}$.

Now let $k_{x}\in K$ and $m\in H$, and consider $k_{x}^{-1}(m(k_{x}(\ast
)))=k_{x}^{-1}(m(x))$. Again let $\tau $ be any $E_{R}^{\ast }$-chain from $%
\ast $ to $x$ and let $\eta $ be an $E_{R}^{\ast }$-chain from $x$ to $m(x)$%
. Let $x^{\prime }:=k_{x}^{-1}(\ast )$; so for some $g_{x^{\prime }}\in G$, $%
g_{x^{\prime }}(\ast )=x^{\prime }$. By definition, $k_{x}^{-1}(m(x))$ is
the endpoint of the lift of $g_{x^{\prime }}(g(\tau \ast \eta ))$ at $\ast $%
. But note that since $m\in H$, $g(\eta )$ is a loop, and therefore so is $%
g_{x}^{\prime }(g(\eta ))$. Since the lift of $g_{x^{\prime }}(g(\tau ))$ to 
$Z$ at $x^{\prime }$ ends at $\ast $, the lift of $g_{x}^{\prime }(g(\eta ))$
to $Z$ at $\ast $ ends at $k_{x}^{-1}(m(x))=k_{x}^{-1}(m(k_{x}(\ast )))$.
But since $g_{x}^{\prime }(g(\eta ))$ is a loop, this shows that $%
k_{x}^{-1}(m(k_{x}(\ast )))=\ast $ and therefore $k_{x}^{-1}mk\in H$,
completing the proof that $H$ is normal in $K$.

Next note that since each of $f$ and $g$ is bi-uniformly continuous, so is $%
h $. We next show that the point preimages of $h$ are the orbits of $K$,
showing that $h:Z\rightarrow X=Z/K$ is a quotient map. Since $h\circ g_{x}=h$
for every $x\in h^{-1}(\ast )$, the orbits of $K$ are contained in point
pre-images. Conversely, suppose that $h(z)=h(w)$ and let $\beta $ be an $%
E_{R}^{\ast }$-chain from $\ast $ to $z$. Since $h(z)=h(w)$ there is some $%
j\in G$ such that $j(g(z))=g(w)$; let $x$ be the endpoint of the unique lift
of $j(g(\overline{\beta }))$ to $Z$ at $w$. By definition of $g_{x}$, $%
g_{x}(z)$ is precisely the endpoint of the unique lift of $j(g(\beta _{z}))$
to $Z$ at $x$. But that lift is precisely $w$, i.e. $g_{x}(z)=w$, completing
the proof that $h$ is a quotient map.

Now define $\theta :K\rightarrow G$ by $\theta (k_{x})=j_{x}$. The
compatibility Equation (\ref{compat}) now becomes 
\begin{equation*}
\theta (k_{x})(g(z))=g(k_{x}(z))
\end{equation*}%
for every $k_{x}\in K$ and $z\in Z$. To see that $\theta $ is a surjective
homomorphism with kernel $H$, let $k_{1},k_{2}\in K$ and $z\in Z$. Applying
compatibility a couple of times we have: 
\begin{equation*}
\theta (k_{x_{1}}k_{x_{2}})(g(z))=g(k_{x_{1}}(k_{x_{2}}(z)))=\theta
(k_{x_{1}})(g(k_{x_{2}}(z))=\theta _{k_{1}}(\theta _{k_{2}}(g(z)))
\end{equation*}%
If $j\in G$, let $x\in g^{-1}(j(\ast ))\subset h^{-1}(\ast )$. Then $%
g(k_{x}(\ast ))=g(x)=j(\ast )$ and therefore by definition $\theta (k_{x})=j$%
. Similarly, if $k_{x}\in \ker \theta $ then $g(k_{x}(\ast ))=\ast $ and
therefore $k_{x}\in H$.

At this point the only fact that we have used about $G$ is that it acts
freely and isomorphically on $Y$. For the next steps we will assume that $H$
acts discretely and impose additional conditions on $E$ for two cases: $G$
acts discretely and $G$ acts prodiscretely. Suppose that $G$ acts
discretely. Then we may take the entourage $E$ from above to have the
property that $N_{E}(G)=1$. Suppose that $(z,k(z))\in E_{R}^{\ast }$ for
some $z\in Z$ and $k\in K$. Then $E$ contains $(g(z),g(k(z)))=(g(z),\theta
(k)(g(z)))$ and therefore $\theta (k)=1$. This means that $k\in H$. But now
since $(z,k(z))\in E_{R}^{\ast }$, $k=1$. That is, $N_{E_{R}^{\ast }}(K)=1$.

Now suppose that $G$ acts prodiscretely (and still $H$ acts discretely)
choose an entourage $F\subset E$ such that $N_{F}(G)\subset U_{E}(G)$. We
will prove that 
\begin{equation}
N_{F_{R}^{\ast }}(K)\subset U_{E_{R}^{\ast }}(K)\text{,}  \label{good}
\end{equation}%
proving that in this case $K$ acts prodiscretely. Suppose that $k_{x}\in
N_{F_{R}^{\ast }}(K)$. That is, $k_{x}=s_{1}\cdot \cdot \cdot s_{n}$, where $%
s_{i}\in S_{F_{R}^{\ast }}(K)$, which in turn means $(s_{i}(z_{i}),z_{i})\in
F_{R}^{\ast }$ for some $z_{i}\in Z$. Therefore $(\theta
(s_{i})(g(z_{i})),g(z_{i}))=(g(s_{i}(z_{i})),g(z_{i}))\in F$, which implies
that each $\theta (s_{i})\in S_{F}(G)$ and therefore $\theta (k_{x})\in
N_{E}(G)\subset U_{E}(G)$. Put another way,%
\begin{equation}
\theta (N_{F_{R}^{\ast }}(K))\subset U_{E}(G)\text{.}  \label{uch}
\end{equation}%
Now suppose that $z\in Z$. We have $(g(z),j_{x}(g(z)))=(g(z),g(k_{x}(z)))\in
E$. Since the restriction of $g$ to $B(z,E_{R}^{\ast })$ is a bijection onto 
$B(g(z),E)$, there is some $w\in B(z,E_{R}^{\ast })$ such that $%
g(w)=j_{x}(g(z))$. From Equation (\ref{essential}), we have: 
\begin{equation*}
k_{x}(z)=g_{k_{x}(w)}^{-1}(j_{x}(g(z)))=g_{k_{x}(w)}^{-1}g(w)=w
\end{equation*}%
That is, $(z,k_{x}(z))\in E_{R}^{\ast }$. Since $z$ was arbitrary, $k_{x}\in
U_{E_{R}^{\ast }}(K)$.

We will now show that $K$ is complete. Since the spaces, hence the groups,
are all metrizable, we may use Cauchy sequences to verify completeness.
Suppose that $\{k_{x_{i}}\}$ is a Cauchy sequence in $K$; then for every $E$%
, $k_{x_{i}}^{-1}k_{x_{j}}\in U_{E_{R}^{\ast }}(K)$ for all large $i,j$.
Applying this to $\ast $, we have that $(k_{x_{i}}^{-1}k_{x_{j}}(\ast ),\ast
)\in E_{R}^{\ast }$ for all large $i,j$. Since $E_{R}^{\ast }$ is invariant,
this means that $(x_{i},x_{j})\in E_{R}^{\ast }$ for all large $i,j$. That
is, $\{x_{i}\}$ is itself Cauchy. Therefore $\{g(x_{i})\}$ is a Cauchy
sequence in the (complete) orbit $G\ast $ of $\ast $. That is, $%
g(x_{i})\rightarrow y\in Y$. Since $g$ is bi-uniformly continuous, this in
turn implies that for every open set $U$ containing the orbit $Hz$ of $z\in
g^{-1}(y)$, there is some $x_{i}\in U$. Since $Hz$ is uniformly discrete, we
may take arbitrarily small open sets containing only single points of $Hz$.
Since $\{x_{i}\}$ is Cauchy, one open set about some particular $w\in Hz$
must contain all but finitely many points of $\{x_{i}\}$. In other words, $%
x_{i}\rightarrow w$. We claim that $k_{x_{i}}\rightarrow k_{w}$, which will
complete the proof that $K$ is complete. In fact, for any sufficiently small
entourage $F$, $(x_{i},w)\in F_{R}^{\ast }$. But this means that $\left(
k_{x_{i}},k_{w}\right) \in N_{F_{R}^{\ast }}(K)$, which is contained in $%
U_{E_{R}^{\ast }}(K)$ by Inclusion (\ref{good}).

For the final case, in which both $G$ and $H$ are complete and act
prodiscretely, choose a countable nested basis $E_{i}$ for $Y$ and take the
resulting resolution of $g$ (see Definition \ref{reso}), which consists of
compatible inverse systems $\{Z_{i},f_{ij}\}_{i,j\in \mathbb{N}}$ of uniform
spaces and $\{G_{i},\theta _{ij}\}_{i,j\in \mathbb{N}}$ of groups $G_{i}$ of
uniform homeomorphisms acting discretely on $Z_{i}$. Let $%
g_{i}:Z_{i}\rightarrow Y=Z_{i}/G_{i}$ be the quotient map, which is a
discrete covering map, and $h_{i}:=f\circ g_{i}$, which by the special case
we already proved, is equivalent to the quotient map $h_{i}:Z_{i}\rightarrow
X=Z_{i}/K_{i}$ via some complete group $K_{i}$ that acts prodiscretely on $%
Z_{i}$ and contains $H_{i}$ as a normal subgroup. That is, we have an
inverse system of quotients via isomorphic actions in the sense of \cite{PQ}%
, and by Proposition 41 of \cite{PQ}, the resulting inverse limit action of $%
K=\underleftarrow{\lim }K_{i}$ on $Z=Z_{i}$ is a generalized regular
covering map. Since $H_{i}$ is a normal subgroup of $K_{i}$, $H$ is a normal
subgroup of $K$.
\end{proof}

\begin{example}
\label{covex}As is well-known, there are traditional regular covering maps
between path and locally path connected spaces whose composition is not even
a covering map; there is a nice illustration of this in \cite{Br}. But
according to Theorem \ref{uniformize} we can \textquotedblleft
uniformize\textquotedblright\ each of these covering maps as discrete
covering maps! This does not contradict Theorem \ref{compo} because as can
be readily discerned from the picture in \cite{Br}, any attempt to
uniformize these two maps creates non-equivalent uniform structures on the
middle space--i.e. the spaces involved can never be simultaneously
uniformized in a way that the result is a composition of discrete covers.
\end{example}

\section{Properties of the Compact Universal Cover}

\begin{proof}[Proof of of Theorem \protect\ref{dir}]
Starting with a basepoint $\ast $ in $X_{0}:=X$ we may choose basepoints $%
\ast $ in each $X_{i}$ so that $f_{i}$ is basepoint-preserving. When $%
f_{i}\leq f_{j}$, by definition this means there is a (uniformly) continuous
surjection $f_{ij}:X_{j}\rightarrow X_{i}$ such that $f_{j}=f_{i}\circ
f_{ij} $. \textit{A priori}, $f_{ij}$ may not be basepoint-preserving, but $%
f_{ij}(\ast )=\ast ^{\prime }$ with $\ast ^{\prime }\in f_{i}^{-1}(\ast )$.
Therefore there is some $g\in G_{i}$ such that $g(\ast ^{\prime })=\ast $,
and replacing $f_{ij}$ by $g\circ f_{ij}$ we may assume that $f_{ij}$ is
basepoint-preserving. By Theorem \ref{eqone}, $f_{ij}$ is a discrete
covering map, and by Proposition \ref{uni!} is the unique
basepoint-preserving such discrete covering map. Suppose that $f_{i}\leq
f_{j}$ and $f_{j}\leq f_{i}$. Then $\iota _{j}:=f_{ij}\circ
f_{ji}:X_{j}\rightarrow X_{j}$ is basepoint-preserving such that $f_{j}\circ
\iota _{j}=f_{j}$. But the identity map on $X_{j}$ has the same properties
and therefore by Proposition \ref{uni!} $\iota _{j}$ is the identity. The
same is true for $i$, and therefore $f_{ij}$ and $f_{ji}$ are inverses. In
particular, $f_{i}$ and $f_{j}$ are equivalent, so $\leq $ is a partial
order.

To see that the set is directed, let $E$ be evenly covered by both $f_{i}$
and $f_{j}$. According to Theorem \ref{quoteq}, for $m=i,j$ $f_{m}$ is
equivalent to the induced quotient map $\pi :X_{E}/K_{m}\rightarrow
(X_{E}/K_{m})/(\pi _{E}(X)/K_{m})$ for normal subgroups $K_{i},K_{j}$ of $%
\pi _{E}(X)$ of finite index. Let $K:=K_{i}\cap K_{j}$ and consider the
induced quotient $f:X_{E}/K\rightarrow (X_{E}/K)/(\pi _{E}(X)/K)=X$, which
by Lemma \ref{addon} is a discrete covering map. We claim that $Y:=X_{E}/K$
is compact. In fact, let $\{y_{i}\}$ be a Cauchy sequence in $Y$; so $%
f(y_{i})\rightarrow x$ for some $x\in X$. Taking a tail of the sequence if
necessary we can assume that $\{f(y_{i})\}$ lies in an open set $U$
containing $x$ that is evenly covered by sets $U_{i}$ in $Y$. Fixing one $%
U_{k}$ and taking preimages in $U_{k}$ we have a convergent subsequence $%
y_{i}^{\prime }\rightarrow x^{\prime }$ for some $x^{\prime }\in f^{-1}(x)$,
with $f(y_{i}^{\prime })=f(y_{i})$. For each $i$ there is some $k_{i}\in \pi
_{E}(X)/K$ such that $k_{i}(y_{i}^{\prime })=y_{i}$. Since $\pi _{E}(X)/K$
is finite, by taking a subsequence if necessary we can assume that $%
y_{i}=k(y_{i}^{\prime })$ for all $i$ and some fixed $k\in G$. But $k$ is a
homeomorphism and since $\{y_{i}^{\prime }\}$ is convergent, so is $%
\{y_{i}\} $.

By Proposition \ref{cpt}, the restriction $f_{C}$ of $f$ to the component $C$
of the basepoint in $X_{E}/K$ is a cover of $X$ by a continuum. Since $%
K\subset K_{m}$ for $m=i,j$, we have induced quotients $g_{m}:Y=X_{E}/K%
\rightarrow X_{m}=(X_{E}/K)/(K_{m}/K)$. Then the restriction $g_{m}^{\prime
} $ of $g_{m}$ to $C$ is a uniformly continuous surjection such that $%
f_{C}=f_{m}\circ g_{m}^{\prime }$. That is, $f_{m}\leq f_{C}$ for $m=i,j$.

By uniqueness, if $f_{i}\leq f_{j}\leq f_{k}$ then $f_{ik}=f_{ij}\circ
f_{jk} $ and therefore we have an inverse system with bonding maps $f_{ij}$.
From Theorem \ref{factor}, the restrictions of the induced quotients $\pi
_{E}(X)/K\rightarrow X$ from the prior paragraph form a countable cofinal
subsystem. It is a classical result (and easy to prove by iteration) that in
this situation the fact that the bonding maps are surjective implies that
the projection maps from the inverse limit are also surjective. The inverse
limit of compact (Hausdorff), connected spaces with surjective bonding maps
is compact and connected. This is mentioned for example in Section 2 of \cite%
{Ch} (in that paper \textquotedblleft continua\textquotedblright\ are not
assumed to be metrizable but only Hausdorff). Moreover, since the spaces in
the (countable!) inverse system are metrizable, so is the inverse limit $%
\widehat{X}$. That is, $\widehat{X}$ is a continuum.

For the homomorphisms $\theta _{ij}:G_{j}\rightarrow G_{i}$ we take the
homomorphism from Theorem \ref{eqone}. By definition, $\theta _{ik}=\theta
_{ij}\circ \theta _{jk}$ and compatibility is simply Theorem \ref{eqone}.1c.
From compatibility and the fact that each $f_{ij}$ is a discrete cover it
follows from Theorem 44 in \cite{PQ} that the projection $\widehat{\phi }:%
\widehat{X}\rightarrow X$ is a generalized universal cover with deck group $%
\pi _{P}(X)=\underleftarrow{\lim }G_{i}$. Proposition \ref{cpt} implies that
every $G_{i}$ is finite, so by definition $\pi _{P}(X)$ is profinite.
\end{proof}

\begin{proof}[Proof of Theorem \protect\ref{corres}]
If $K$ is a closed normal subgroup of $\pi _{P}(X)$ of finite index then
since $\pi _{P}(X)$ is complete, $K$ is complete. By Proposition \ref{haus}, 
$\widehat{X}/K$ is Hausdorff, hence a continuum that covers $X$ via the
induced quotient.

For the converse, suppose first that $f$ is a discrete cover, so $Y=X_{j}$
for some $j$ and $f=f_{j}:X_{j}\rightarrow X=X_{j}/G_{j}$. Define $%
K_{f_{j}}:=\ker \theta ^{j}$, where $\theta ^{j}:\pi _{P}(X)=\underleftarrow{%
\lim }G_{i}\rightarrow G_{j}$ is the projection. We claim that the
projection $\widehat{\phi }^{j}:\widehat{X}\rightarrow X_{j}$ is equivalent
to the quotient $\widehat{X}\rightarrow \widehat{X}/K_{f}$. Since $\widehat{%
\phi }^{j}$ is surjective (see the proof of Theorem \ref{dir}) hence
bi-uniformly continuous, we need only show that $\widehat{\phi }^{j}(x)=%
\widehat{\phi }^{j}(y)$ if and only if for some $g\in K_{f_{j}}$, $g(x)=y$.
If $\widehat{\phi }^{j}(x)=\widehat{\phi }^{j}(y)$ then $\widehat{\phi }(x)=%
\widehat{\phi }(y)$ and therefore there is some $g\in \pi _{P}(X)$ such that 
$g(x)=y$. By compatibility, 
\begin{equation*}
\theta ^{j}(g)(\widehat{\phi }^{j}(x))=\widehat{\phi }^{j}(g(x))=\widehat{%
\phi }^{j}(y)=\widehat{\phi }^{j}(x)
\end{equation*}%
and since the action is free, $g\in \ker \theta ^{j}=K_{f_{j}}$. Conversely,
suppose there is some $g\in K_{f_{j}}$ such that $g(x)=y$. Then 
\begin{equation*}
\widehat{\phi }^{j}(y)=\widehat{\phi }^{j}(g(x))=\theta ^{j}(g)(\widehat{%
\phi }^{j}(x))=\widehat{\phi }^{j}(x)\text{.}
\end{equation*}%
We may now invoke Proposition \ref{inducedp} to see that $f_{j}$ is
equivalent to the induced quotient $\pi :\widehat{X}/K_{f_{j}}\rightarrow
X=\left( \widehat{X}/K_{f_{j}}\right) /(\pi _{P}(X)/K_{f_{j}})$, and by
Proposition \ref{cpt}, $G_{j}=\pi _{P}(X)/K_{f_{j}}$ is finite, and
therefore $K_{f_{j}}$ has finite index.

For $f$ an arbitrary generalized regular covering map we will utilize the
resolution of $f$ (see Definition \ref{reso}). Since the covers in the
resolution are discrete, they constitute a directed subsystem of the inverse
system of all coverings of $X$ by compacta indexed by some set $J$, $%
\{X_{j},f_{ij}\}_{j\in J}$ where by the previous part, for each $j$ we have
the projection $f^{j}:\widehat{X}\rightarrow \widehat{X}/K_{f_{j}}=X_{j}$
and $f_{j}:X_{j}\rightarrow X$ is equivalent to the quotient 
\begin{equation*}
\pi _{j}:X_{j}=\widehat{X}/K_{f_{j}}\rightarrow X=\left( \widehat{X}%
/K_{f_{j}}\right) /\left( \pi _{P}(X)/K_{f_{j}}\right) =\widehat{X}/\pi
_{P}(X))\text{.}
\end{equation*}%
Note that each $K_{j}$, being a subgroup of $\pi _{P}(X)$, acts freely and
has compact orbits. Therefore the proof is complete by Theorem \ref{snark}.
\end{proof}

\begin{proof}[Proof of Theorem \protect\ref{main}]
For existence we may simply take the quotient map $f_{L}:\widehat{X}%
\rightarrow \widehat{X}/K_{f}$ from the Galois Correspondence. Uniqueness
(mod $G$) follows from Proposition \ref{uni!}.
\end{proof}

\begin{proof}[Proof of Theorem \protect\ref{unique}]
If $\widehat{X}$ were not compactly simply connected then there would be
some non-trivial discrete cover $f:Y\rightarrow \widehat{X}$ by a continuum.
By Theorem \ref{compo}, $h:=\widehat{\phi }\circ f$ is a generalized
universal cover of $X$. Now Theorem \ref{main} implies that there is a
unique (mod $\pi _{P}(X)$) generalized covering map $h_{L}:\widehat{X}%
\rightarrow Y$ such that $\widehat{\phi }=h\circ h_{L}$. This implies that $%
g:=f\circ h_{L}:\widehat{X}\rightarrow \widehat{X}$ satisfies 
\begin{equation*}
\widehat{\phi }\circ g=\widehat{\phi }\circ f\circ h_{L}=h\circ h_{L}=%
\widehat{\phi }\text{.}
\end{equation*}%
But this makes $g$ a lift of $\widehat{\phi }$ itself. Since the identity
map is also a lift of $\widehat{\phi }$, by uniqueness of Theorem \ref{main}%
, $g$ is equal (mod $\pi _{P}(X)$) to the identity. Since $h_{L}$ is onto,
this means that $f$ is 1-1, a contradiction.

Now suppose that $f:Y\rightarrow X$ is a generalized regular covering map
with $Y$ a compactly simply connected continuum. By the Universal Property
there is a generalized regular covering map $f_{L}:\widehat{X}\rightarrow Y$
such that $f\circ f_{L}=\widehat{\phi }$. But if $f_{L}$ were not a uniform
homeomorphism then $Y$ would have a nontrivial discrete cover by a
continuum, meaning that $\pi _{P}(Y)\neq 1$, a contradiction.
\end{proof}

\begin{proof}[Proof of Corollary \protect\ref{ecor}]
By Theorem \ref{compo}, $g:=f\circ \widehat{\phi _{Y}}:\widehat{Y}%
\rightarrow X$ is a generalized regular covering map and therefore by the
Universal Property there is a lift $g_{L}:\widehat{X}\rightarrow \widehat{Y}$
such that $g\circ g_{L}=\widehat{\phi _{X}}$. On the one hand, the Universal
Property gives us a unique (mod $G$) lift $f_{L}:\widehat{X}\rightarrow Y$
such that $f\circ f_{L}=\widehat{\phi _{X}}$. Applying it again, there is a
generalized regular covering map $(f_{L})_{L}:\widehat{Y}\rightarrow 
\widehat{X}$ such that $f_{L}\circ (f_{L})_{L}=\widehat{\phi _{Y}}$. Now $%
(f_{L})_{L}\circ g_{L}:\widehat{X}\rightarrow \widehat{X}$ is a generalized
regular covering map such that 
\begin{equation*}
\widehat{\phi _{X}}\circ ((f_{L})_{L}\circ g_{L})=f\circ f_{L}\circ
(f_{L})_{L}\circ g_{L}=f\circ \widehat{\phi _{Y}}\circ g_{L}=g\circ g_{L}=%
\widehat{\phi _{X}}\text{.}
\end{equation*}%
By uniqueness again this means that $(f_{L})_{L}\circ g_{L}$ is (mod $\pi
_{P}(X)$) equal to the identity. Since $(f_{L})_{L}$ is onto, $g_{L}$ must
be 1-1 and hence must be a uniform homeomorphism.
\end{proof}

\begin{example}
\label{solenoid}As a final example we will consider what happens with the $2$%
-adic solenoid $X=\Sigma _{2}$, which is the inverse limit of the $2^{n}$
covers of the circle. We already saw in \cite{BPCG} that $\widehat{X}$
(viewed as a compact group) is the \textquotedblleft universal
solenoid\textquotedblright , which is the inverse limit of all covers of the
circle by itself, which is precisely how $\widehat{X}$ is defined in the
present paper. In \cite{BPUU} we saw that $\widetilde{X}=\mathbb{R}$ using
the map induced by the generalized universal covering of the circle by $X$.
But it is perhaps useful to examine the two inverse systems themselves and
their interaction, beginning with the fundamental inverse system. As is well
known, $X$ is locally a product between a compact real interval and the
Cantor set $C$. In particular, we may take for a basis of the unique uniform
structure on $X$ all entourages $E$ such that every $B(x,E)$ is homeomorphic
to $(-\varepsilon ,\varepsilon )\times V$, where $V$ is an open set in $C$.

Now the basepoint $\ast $ lies on some path component, which is the image of
a continuous 1-1 map $p:\mathbb{R\rightarrow }X$ with $p(0)=\ast $. As $t$
increases, the local product structure requires that $p(t)$ leave the local
product neighborhood, but compactness forces it to return as another path
component of the local product around $\ast $. For some $t_{1}^{\ast }$, $%
p(t_{1}^{\ast })$ has the same second coordinate as $\ast $. Now suppose $%
(\ast ,p(t_{1}^{\ast }))\in E$. Then $E$ doesn't \textquotedblleft see the
gap\textquotedblright\ between $\ast $ and $p(t_{1}^{\ast })$ and there is
an $E$-loop $\lambda $ at $\ast $ that starts as a subdivision of the
segment $p_{1}$ of $p$ restricted to $[0,t_{1}^{\ast }]$ and then jumps from 
$p(t_{1}^{\ast })$ to $\ast $. If $E$ is small enough in the
\textquotedblleft real direction\textquotedblright\ then $\lambda $ cannot
be $E$-null and so represents a non-trivial element of $\pi _{E}(X)$. That
is, $\lambda $ \textquotedblleft unrolls\textquotedblright\ to its lift $%
\widetilde{\lambda }$ to $X_{E}$ at the basepoint, ending at $[\lambda
]_{E}\in \phi _{E}^{-1}(\ast )\backslash \{\ast \}$ in $X_{E}$. Now consider
the lift of the segment $p_{1}$ to $X_{E}$ at the basepoint, which contains $%
\widetilde{\lambda }$ except for its endpoint $[\lambda ]_{E}$. Due to the
gap that is crossed by $\lambda $, $[\lambda ]_{E}$ lies in some path
component that does not contain the basepoint in $X_{E}$. The segment $p_{1}$
\textquotedblleft unrolls\textquotedblright\ in terms of the uniform
structure, but not of course topologically because it is simply connected
and it must remain inside the path component of $\ast $ in $X_{E}$. That is,
the non-trivial deck group element $[\lambda ]_{E}$ takes the path component
of the basepoint in $X_{E}$ to a different path component.

Returning to the path component of $\ast $ in $X$, the lift of $p$ at $\ast $
can never return to any $E^{\ast }$-ball because $\phi _{E}$ is a bijection
on $E^{\ast }$-balls. This implies that $X_{E}$ is not compact (although it
is locally compact). $X_{E}$ is also not connected, and in fact for
sufficiently small $F\subset E$, $\phi _{EF}$ is not surjective. That is, as
soon as an entourage $F$ \textquotedblleft sees a gap that $E$ didn't
see\textquotedblright\ the loop $\lambda $ can no longer be refined in its $%
E $-homotopy class to an $F$-loop. Therefore $\theta _{EF}:\pi
_{F}(X)\rightarrow \pi _{E}(X)$ is not surjective, hence $\phi _{EF}$ is not
surjective either. This means that some path components of $X_{E}$ are not
in the image of $\phi ^{E}:\widetilde{X}\rightarrow X_{E}$ and in the
inverse limit, only the path component of $\ast $ \textquotedblleft
survives\textquotedblright . That is, the lift of $p$ to $\widetilde{X}$ at $%
\ast $ is equal to $\widetilde{X}$ and $\phi :\widetilde{X}\rightarrow X$ is
a bijection onto the path component. Note that $\phi $ in this case is
uniformly continuous but not bi-uniformly continuous--it \textquotedblleft
rolls up\textquotedblright\ $\mathbb{R}$ into the dense path component of $%
\ast $.

Now any discrete cover of $X$ by a continuum $X_{i}$ is covered by some $%
X_{E}$, which maps into $X_{i}$ as a dense subset--and therefore $X_{i}$ is
a compactification of $X_{E}$. The cover $f_{i}:X_{i}\rightarrow X$ is
induced by a quotient of $X_{E}$ by a normal subgroup of finite index and of
course by definition $f_{i}$ is surjective. In other words, the inverse
system leading to $\widehat{X}$ can be viewed as a uniquely determined
compactification of the fundamental inverse system, which via compactness
\textquotedblleft corrects\textquotedblright\ the problem that the bonding
maps of the fundamental inverse system are not surjective in this particular
example for any choice of a cofinal sequence of entourages. Some of this
analysis applies more generally to matchbox manifolds.
\end{example}

\end{document}